\documentclass{article}
\usepackage[utf8]{inputenc}
\usepackage{amsmath}
\usepackage{amssymb}
\usepackage{amsfonts}
\usepackage{amsthm}
\usepackage{setspace}
\usepackage{bbm}
\usepackage[english]{babel}
\usepackage{graphicx}
\usepackage{xcolor}
\usepackage{fancyhdr} 
\usepackage[left=1in,right=1in,top=1in,bottom=1in]{geometry}
\usepackage{amsmath,amscd}
\usepackage[all]{xy}
\usepackage{tikz}
\usepackage{tikz-cd}
\usetikzlibrary{shapes}
\usetikzlibrary{arrows}
\usetikzlibrary{matrix}
\usetikzlibrary{calc}
\usetikzlibrary{positioning}
\usetikzlibrary{decorations.pathreplacing}
\usetikzlibrary{decorations.markings}
\usetikzlibrary{intersections}
\usepackage{mathtools}
\usepackage[hypcap=true]{caption}
\usepackage{hyperref}
\usepackage[numbers]{natbib}
\usepackage{authblk}
\DeclareFontFamily{U}{mathx}{\hyphenchar\font45}
\DeclareFontShape{U}{mathx}{m}{n}{
      <5> <6> <7> <8> <9> <10>
      <10.95> <12> <14.4> <17.28> <20.74> <24.88>
      mathx10
      }{}
\DeclareSymbolFont{mathx}{U}{mathx}{m}{n}
\DeclareFontSubstitution{U}{mathx}{m}{n}
\DeclareMathAccent{\widecheck}{0}{mathx}{"71}
\usepackage{marvosym}
\usepackage{mathrsfs}
\usepackage{theoremref}

\newcommand{\tc}[2]{{\textcolor{#1}{#2}}}
\newcommand{\lrp}[1]{\left(#1\right)}
\newcommand{\lrb}[1]{\left[#1\right]}

\newcommand{\lrm}[1]{\left|#1\right|}
\newcommand{\lrc}[1]{\left\{#1\right\}}
\newcommand{\lra}[1]{\left\langle{#1}\right\rangle}

\newcommand{\LRA}{\Longleftrightarrow}

\newcommand{\eq}[2]{\begin{equation}\label{#2} \begin{split} #1  \end{split} \end{equation}}
\newcommand{\eqn}[1]{\begin{equation*} \begin{split} #1 \end{split} \end{equation*}}

\newcommand{\Q}{\mathbb{Q} }
\newcommand{\R}{\mathbb{R} }

\newcommand{\Z}{\mathbb{Z} }

\newcommand{\kk}{\mathbbm{k} }

\newcommand{\vb}[1]{\mathbf{#1}}
\newcommand{\cA}{\mathcal{A} }

\newcommand{\ssO}{\mathcal{O} }

\newcommand{\trop}{\mathrm{trop} }

\newcommand{\mf}[1]{\mathfrak{#1} }
\newcommand{\tf}{\vartheta}
\newcommand{\wall}{\mathfrak{d}}
\newcommand{\scat}{\mathfrak{D}}
\makeatletter
\newlength{\negph@wd}
\DeclareRobustCommand{\negphantom}[1]{%
  \ifmmode
    \mathpalette\negph@math{#1}%
  \else
    \negph@do{#1}%
  \fi
}
\newcommand{\negph@math}[2]{\negph@do{$\m@th#1#2$}}
\newcommand{\negph@do}[1]{%
  \settowidth{\negph@wd}{#1}%
  \hspace*{-\negph@wd}%
}
\makeatother

\makeatletter
\newcommand{\subalign}[1]{%
  \vcenter{%
    \Let@ \restore@math@cr \default@tag
    \baselineskip\fontdimen10 \scriptfont\tw@
    \advance\baselineskip\fontdimen12 \scriptfont\tw@
    \lineskip\thr@@\fontdimen8 \scriptfont\thr@@
    \lineskiplimit\lineskip
    \ialign{\hfil$\m@th\scriptstyle##$&$\m@th\scriptstyle{}##$\hfil\crcr
      #1\crcr
    }%
  }%
}
\makeatother

\definecolor{blue-ish}{rgb}{0,.55,.7}
\definecolor{more-blue}{rgb}{0,.2,.8}
\definecolor{more-green}{rgb}{.1,.5,.1}

\theoremstyle{plain}
\newtheorem{theorem}{Theorem}
\newtheorem{prop}[theorem]{Proposition}
\newtheorem{lemma}[theorem]{Lemma}

\newtheorem{cor}[theorem]{Corollary}

\theoremstyle{definition}
\newtheorem{definition}[theorem]{Definition}

\newtheorem{notation}[theorem]{Notation}

\theoremstyle{remark}
\newtheorem{remark}[theorem]{Remark}

\DeclareMathOperator{\supp}{supp}

\DeclareMathOperator{\conv}{conv}
\DeclareMathOperator{\bconv}{conv_{BL}}

\DeclareMathOperator{\Cone}{Cone}
\newcommand{\sumt}{\displaystyle{\sum}_\tf}
\DeclareMathOperator*{\Sumt}{\sumt}
\DeclareMathOperator{\+t}{+_\tf}
\DeclareMathOperator{\codim}{codim}
\DeclareMathOperator{\colim}{colim}
\DeclareMathOperator{\Sing}{Sing}

\title{Fundamentals of Broken Line Convex Geometry}

\author[1]{Juan Bosco Fr\'ias-Medina}
\author[2]{Timothy Magee}
\affil[1]{Universidad Michoacana de San Nicol\'as de Hidalgo, \href{mailto:juan.frias@umich.mx}
{juan.frias@umich.mx}}
\affil[2]{Hollins University,  \href{mailto:mageetd@hollins.edu}{mageetd@hollins.edu}}
\date{}

\begin{document}

\maketitle

\begin{abstract}
    We develop the fundamentals of a new theory of convex geometry -- which we call {\it{broken line convex geometry}}.
    This is a theory of convexity where the ambient space is the rational tropicalization of a cluster variety, as opposed to an ambient vector space. 
    In this theory, {\it{line segments}} are replaced by {\it{broken line segments}}, and we adopt the  notion of convexity in \cite{BLC}.
    We state and prove broken line convex geometry versions of many standard results from usual convex geometry.
\end{abstract}

\tableofcontents

\section{Introduction}

In this paper we study {\it{broken line convex geometry}} -- a generalization of convex geometry in which the ambient space is the tropicalization of a cluster variety rather than simply a vector space, and in which broken line segments play the role ordinarily filled by line segments.
We show that many classical convex geometry results remain true in this setting.
For instance, versions of the following classical results remain true in broken line convex geometry:
\begin{enumerate}
    \item \label{it:DilateSumConvex}A set $S$ is convex if and only if $t S +(1-t)S = S$ for all $t \in [0,1]$.
    \item \label{it:ConvSum-SumConv}$\conv(S+T) = \conv(S) + \conv(T)$.
    \item \label{it:ConvFun-ConvSet}If $\varphi$ is a convex function, then the locus where $\varphi$ is at least some constant $r$ is a convex set. 
    \item A bounded polyhedron is the convex hull of its vertices.
    \item \label{it:StronglyConvex-FullDim} The dual of a convex set $S$ is full dimensional if and only if $S$ is strongly convex.
    \item \label{it:DualFaceComplex}If $P$ and $P^\circ$ are dual polytopes, there is a bijective, containment-reversing correspondence between the faces of $P$ and $P^\circ$.
\end{enumerate}
Other aspects of the theory need a bit of modification, but remain quite pleasant.
In broken line convex geometry, the faces of a polyhedral set are generally not broken line convex.
However, they satisfy a rather natural  weaker convexity notion -- which we call {\it{weak convexity}}.
They may also fail to form a complex.
Nevertheless, they do have a structure reminiscent of a polyhedral complex, forming what we call a {\it{pseudo-complex}}.

As is already evident in the brief list above, one operation central to the theory of convex geometry is the Minkowski sum.
As such, a key element of this story is our notion of Minkowski addition in a tropical space.
It is morally the same as usual Minkowski addition, but the lack of linear structure in tropical spaces makes this addition multi-valued.
See \thref{def:TropicalMinkowskiSum} for the precise definition.
\S\ref{sec:TMS} treats the interplay of this tropical Minkowski sum and the {\it{broken line convex hull}} of \cite{BLC}.
We find that these concepts relate to each other in much the same way as the usual Minkowski sum and convex hull do.
In particular, \thref{char_blc} is the broken line convex geometry version of Item~\ref{it:DilateSumConvex} and \thref{thm:Compatible} is the broken line convex geometry version of Item~\ref{it:ConvSum-SumConv}.

Next, we turn our attention to the meaning of convexity of functions in broken line convex geometry.
Here again we adapt the linear definition to the tropical setting by replacing {\it{the line segment}} between a pair of points with {\it{all broken line segments}} between a pair of points:
\begin{definition}[\thref{def:ConvWRTBL}]
Let $S\subset U^{\trop}(\Q)$ be a broken line convex set. 
A function $\varphi:S\rightarrow \Q $ is \textit{convex with respect to broken lines} if for 
any broken line segment $\gamma:[t_1,t_2]\rightarrow S$, 
we have that    
\eqn{\varphi(\gamma(t)) \geq \lrp{\dfrac{t_2-t}{t_2-t_1}} \varphi(\gamma(t_1)) + \lrp{\dfrac{t-t_1}{t_2-t_1}} \varphi(\gamma(t_2))}
for all $t\in [t_1,t_2]$.
\end{definition}
We then give an equivalent characterization these functions in terms of structure constants of $\tf$-function multiplication:\footnote{We will discuss these $\tf$-functions and structure constants in greater detail in \S\ref{sec:Background}. For now, a non-zero structure constant $ \alpha_{p_1, \dots, p_d}^{q} $ means that $q$ is a value of the multi-valued sum $p_1 \+t \dots \+t p_d$.}
\begin{theorem}[\thref{prop:ConvWRTBL}, \thref{rem:ConvWRTBL-Equiv}]
    Let $S\subset U^{\trop}(\Q)$ be broken line convex.
    Then $\varphi:S \to \Q$ is convex with respect to broken lines if and only if
    for all $s_1,\dots, s_d, s \in S$, $a_1, \dots, a_d \in \Q_{\geq 0}$ with $a_1 s_1,\dots, a_d s_d$, and $(a_1+\cdots +a_d)s$ all integral, and $\alpha_{a_1 s_1,\dots, a_d s_d}^{(a_1+\cdots +a_d)s} \neq 0$, we have
    \eqn{\varphi(s) \geq \sum_{i=1}^{d} \frac{a_i}{a_1+\cdots +a_d}\varphi(s_i) .}
\end{theorem}
Note that the inequality appearing here is a broken line convex geometry version of Jensen's inequality.  
We use these equivalent characterizations to prove the equivalence of Gross-Hacking-Keel-Kontsevich's {\it{min-convex}} and {\it{decreasing}} definitions.  See \thref{cor:min-convex=decreasing}.
We then describe other properties of functions which are convex with respect to broken lines.
In particular, \thref{prop:ConvWRTBL-BLC} is the broken line convex geometry version of Item~\ref{it:ConvFun-ConvSet}.

After this discussion of convexity for functions, we treat {\it{polyhedral}} broken line convex geometry.
The canonical pairing between tropicalizations of mirror cluster varieties affords us a natural notion of a half-space in this context.
\begin{definition}[\thref{def:Half-space-Hyperplane}]
For $y\in (U^\vee)^\trop(\Q)$ and $r\in \Q$, we call the set \eqn{K(y,r):=\lrc{x \in U^\trop(\Q) : \lra{x, y } \geq -r }} a {\it{tropical half-space}}.
\end{definition}
This in turn provides a natural analogue of a polyhedron-- we say a subset $S\subset U^{\trop}(\Q)$ is {\it{polyhedral}} if it is the intersection of finitely many tropical half-spaces. (See \thref{def:polyhedral}.)
{\it{Faces}} of $S$ are defined much like in usual convex geometry-- we take a tropical half-space containing $S$ and intersect its boundary with $S$.  See \thref{def:face-half-space}.
As mentioned above, these faces satisfy only a weaker notion of convexity.
If we choose a pair of points $x_1,\, x_2$ in a face $F$, we cannot say that $F$ contains {\emph{all}} broken line segments connecting $x_1$ and $x_2$.
We can only say that $F$ contains {\emph{some}} broken line segment connecting $x_1$ and $x_2$. (See \thref{cor:FaceWC}.)
The failure of these faces to be broken line convex hinders another familiar property from usual convex geometry-- the intersection of two faces need not be a face.
As such, faces may not form a complex.  They do however have a structure very reminiscent of a complex, which we refer to as a {\it{pseudo-complex}}.  See \thref{def:pseudo-complex} and \thref{prop:FacePseudo-Complex}.
Moreover, we show in \thref{prop:FaceDual} that the face pseudo-complexes of polar polytopal sets are related in precisely the same way as the linear case described in Item~\ref{it:DualFaceComplex}.


Ultimately, our expectation is that broken line convex geometry will encode the algebraic geometry of minimal models for cluster varieties in much the same way that usual convex geometry encodes the algebraic geometry of toric varieties.
This is the underlying motivation for the theory, and we hope to develop a 
\eqn{\lrc{\textit{polyhedral broken line convex geometry}}\longleftrightarrow\lrc{\textit{algebraic geometry of minimal models for cluster varieties}}} dictionary in future works.

Shortly before posting the first version of this paper, we met another team working in a similar direction at the conference {\it{Mirror Symmetry for Cluster Varieties and Representation Theory}}.
J. B. Fr\'ias-Medina presented on this paper while C. Manon presented his joint work \cite{EHM} with L. Escobar and M. Harada.
While our perspectives differ -- their approach is more rooted in the theory of Newton-Okounkov bodies while ours stems from the theory of mirror symmetry for cluster varieties elucidated in \cite{GHKK} -- there is quite a bit of overlap with the topics we discuss.
Many of the constructions, structures, and questions we discuss are remarkably similar, including the presence of an intrinsic piecewise-linear ambient manifold, notions of ``line segments'' within this manifold, a convexity definition that considers {\emph{all}} line segments between a pair of points, and a multi-valued version of Minkowski addition.
A similar polyhedral theory to the one we discuss in \S\ref{sec:Polyhedra} also appears in the recent paper \cite{Lai-Zhou} in the two dimensional setting, with exciting applications to mirror symmetry for log Calabi-Yau surfaces. 

Finally, we would like to mention that in our initial plan for this paper, we envisioned a closing crescendo in which we state and prove a broken line convex geometry version of Borisov's duality for nef-partitions.  See \cite{Borisov} for the original version.
Unfortunately, the problem has thwarted all of our attempts to date. 
From the outset, we viewed this paper as part of a research program we undertook with our close collaborators L. Bossinger, M.-W. Cheung, and A. N\'ajera Ch\'avez, with the goal of generalizing Batyrev and Batyrev-Borisov mirror symmetry constructions from the setting of Gorenstein Fano toric varieties to the setting of Gorenstein Fano minimal models for cluster varieties.
We hope that the broken line convex geometry results of this paper will be an important step toward that common goal.

\subsection*{Acknowledgments}

We would first like to thank our collaborators L. Bossinger, M.-W. Cheung, and A. N\'ajera Ch\'avez.
As mentioned above, we view this paper as part of a program we have been pursuing with them, and many ideas in this paper have surely arisen in discussion with them over our years of collaboration.
For instance, T.M. remembers discussing a candidate for Minkowski addition in this context (similar to \thref{def:TropicalMinkowskiSum}) with A. N\'ajera Ch\'avez once several years ago.
He also recalls attempting to prove \thref{cor:min-convex=decreasing} by other means with M.-W. Cheung.
J.B.F.-M. acknowledges the grant of ``Programa de Estancias Posdoctorales por M\'exico'' from CONAHCYT (now SECIHTI).
T.M. gratefully acknowledges the support of the EPSRC through the grant EP/V002546/1. We would both like to thank UNAM Campus Morelia and the EPSRC for supporting two research visits during the preparation of this article.
First, UNAM Campus Morelia kindly provided an office for T.M. while he visited J.B.F.-M. in Morelia.
The grant EP/V002546/1 also provided travel funding for this trip.
Next, the grant EP/V002546/1 provided travel funding to J.B.F.-M. for his visit to T.M. and his participation in the {\it{Mirror Symmetry for Cluster Varieties and Representation Theory}} conference, where he presented on this project.
We are very optimistic that the feedback from this presentation will lead to many exciting new projects and collaborations.

\section{Background}\label{sec:Background}

The notion of broken line convexity used in this paper comes from \cite{BLC}, where the main result is the equivalence of this convexity notion with the algebraic notion of {\it{positivity}} from \cite{GHKK}.
That said, we will employ subtly different conventions and definitions here.
First, as we are only ever interested in the rational points of our tropical spaces, we will always work over $\Q$ instead of $\R$.\footnote{The prescient reader may raise concern about placement of basepoints for broken lines in regions of dense walls.  We will address this concern in \S\ref{sec:Scattering-BrokenLines} with a discussion of {\it{broken lines}} vs. {\it{generic broken lines}}, as in \cite{BLC}, with one modification. The sequence of generic broken lines we use to define our broken lines here will live in the finite order scattering diagrams whose colimit produces the cluster scattering diagram.}
Next, the definition of positivity in \cite{GHKK}, and in turn that of broken line convexity in \cite{BLC}, makes reference to {\emph{closed}} sets.
However, these definitions may equally well be made without requiring closure.
Moreover, the proof of the equivalence of broken line convexity and positivity in \cite{BLC} does not rely on closure -- the result still holds if closure is dropped from both definitions.
We do precisely this.

\begin{definition}[\cite{BLC}]\thlabel{def:BLC}
    A subset $S$ of $U^{\trop}(\Q)$ is {\it{broken line convex}} if for every pair of points $s_1$, $s_2$ in $S$, every broken line segment with endpoints $s_1$ and $s_2$ has support entirely contained in $S$.
\end{definition}
Less formally:
\begin{center}
    {\it $S$ is broken line convex if every broken line segment that starts and ends in $S$ stays in $S$.}
\end{center}
This is the natural generalization of usual convexity to $U^{\trop}(\Q)$, where broken line segments fill the role occupied by line segments in usual convex geometry.
The aforementioned {\it{positivity}} which it is equivalent to is defined as follows:

\begin{definition}[\cite{GHKK}]\thlabel{def:positive}
    A subset $S$ of $U^{\trop}(\Q)$ is {\it{positive}} if for any non-negative integers $a$ and $b$, and any integral tropical points $p \in a\, S(\Z)$, $q \in b\, S(\Z)$, and $r\in U^{\trop}(\Z)$ with $\alpha_{p,q}^r \neq 0$, we have $r\in (a+b)S(\Z)$.
\end{definition}

In usual convex geometry, there is a canonical way to take a possibly non-convex set and replace it with a convex set which contains it -- namely the convex hull.
There is a completely analogous procedure here:

\begin{definition}[\cite{BLC}]\thlabel{def:BLCHull}
Let $S \subset U^\trop(\Q)$.
We define the {\it{broken line convex hull of $S$}}, denoted $\bconv(S)$ to be the intersection of all broken line convex sets containing $S$. 
\end{definition}

\begin{notation}
    As we will always work over $\Q$, when we write an interval $[t_1,t_2]$ we mean an interval in $\Q$, {\emph{not}} an interval in $\R$.
\end{notation}

\subsection{Results from \texorpdfstring{\cite{ThetaReciprocity}}{CMMM}}

There are two key results from \cite{ThetaReciprocity} that we will need throughout the course of this work.
We will state them here in simplified form -- the setting of \cite{ThetaReciprocity} is more general than that of the current work.

Let $U$ and $U^\vee$ be mirror cluster varieties for which the full Fock-Goncharov conjecture holds. 
Roughly, this means that the integral tropical points of each parametrize a basis (the $\tf$-basis) for the algebra of regular functions on its mirror.
For a more precise definition, see \cite[Definition~0.6]{GHKK}.\footnote{The terminology refers to a conjecture of Fock and Goncharov in \cite{FG1}.}
Then:

\begin{theorem}[\cite{ThetaReciprocity},``Theta Reciprocity'']\thlabel{thm:ThetaRecipocity}
Let $x\in U^{\trop}(\Z)$ and $y \in (U^\vee)^{\trop}(\Z)$. Then $x(\tf_y)=y(\tf_x)$.
\end{theorem}
This means we have a truly canonical pairing between $U^{\trop}(\Z)$ and $(U^\vee)^{\trop}(\Z)$, rather than two different evaluation pairings.
This pairing $\lra{\ \boldsymbol{\cdot}\ ,\ \boldsymbol{\cdot}\ }$ extends uniquely to a pairing between $U^{\trop}(\Q)$ and $(U^\vee)^{\trop}(\Q)$.

The other key result of \cite{ThetaReciprocity} we need is the {\it{valuative independence theorem}}.

\begin{theorem}[\cite{ThetaReciprocity},``Valuative Independence'']\thlabel{thm:ValuativeIndependence}
Let \eqn{f = \sum_{y\in (U^\vee)^{\trop}(\Z)} c_y \tf_y} be any regular function on $U$ and $x \in U^{\trop}(\Z)$ any integral tropical point.
Then \eqn{x(f) = \min_{c_y \neq 0} \lrc{x(\tf_y)}.}
\end{theorem}

Recall that integral tropical points are discrete valuations, and as such the inequality 
\eqn{x(f) \geq \min_{c_y \neq 0} \lrc{x(\tf_y)}}
holds by definition.
The valuative independence theorem replaces the inequality with an equality, essentially by eliminating the possibility of pole cancellations. 

\subsection{Scattering diagrams and broken lines}\label{sec:Scattering-BrokenLines}

We refer the reader to \cite{GHKK} for background on cluster scattering diagrams. 
We simply recall a few basic points, in part to fix terminology and notation, and follow up with some simple observations.
\begin{enumerate}
    \item \label{it:walls}A scattering diagram $\scat$ is a collection of {\it{walls}} $(\wall,f_\wall)$, where the {\it{support}} $\wall$ is a rational polyhedral cone in the ambient vector space, and $f_\wall$ is the {\it{scattering function}}, which lives in a certain completed monoid ring $\widehat{\kk[P]}$. (See \cite[Definition~1.4]{GHKK}.)
    \item \label{it:wall-crossing}Crossing a wall $(\wall,f_\wall)$ induces an automorphism $\mathfrak{p}_{f_\wall}$ of $\widehat{\kk[P]}$, and a generic path $\gamma$ crossing multiple walls induces an automorphism $\mathfrak{p}_{\gamma}$ of $\widehat{\kk[P]}$ by composition. (See \cite[Definition~1.2 and discussion of {\it{path ordered product}}]{GHKK}.)
    \item \label{it:singular}For $\mathfrak{p}_{\gamma}$ of Item~\ref{it:wall-crossing} to be defined, $\gamma$ must avoid the {\it{singular locus}} of $\scat$:
    \eqn{\Sing(\scat):= {\bigcup_{(\wall,f_\wall)\in \scat} \partial \wall} \quad \cup {\bigcup_{\substack{(\wall_1,f_{\wall_1}),(\wall_2,f_{\wall_2})\in \scat\\ \codim(\wall_1 \cap \wall_2)\geq 2}}\wall_1 \cap \wall_2}.}
    \item \label{it:construction}Cluster scattering diagrams are constructed order by order as a colimit $\scat = \colim_k \scat_k$, and each $\scat_k$ is a finite scattering diagram. (See \cite[Appendix~C]{GHKK}.)
    \item \label{it:mutation-invariance}While the construction mentioned in Item~\ref{it:construction} depends upon a choice of initial seed $\vb{s}$, if $\vb{s}$ and $\vb{s}'$ are related by mutation $\mu:T_{L^*;\vb{s}}\dashrightarrow T_{L^*;\vb{s}'}$, then $\mu^{\trop}(\scat_{\vb{s}})$ is equivalent to $\scat_{\vb{s}'}$. (See \cite[\S1.3]{GHKK}.)
\end{enumerate}

Since the support $\wall$ of each wall is a {\emph{rational}} polyhedral cone (Item~\ref{it:walls}), we can describe the scattering diagram perfectly well over $\Q$ rather than $\R$.
Next, thanks to mutation invariance of the scattering diagram (Item~\ref{it:mutation-invariance}), we interpret $(U^\vee)^{\trop}(\Q)$ as the natural ambient space of the scattering diagram for $U$.
To write down a scattering diagram explicitly however, we choose a seed of the cluster structure.
This selects a cluster torus $T_L$ in $U$ and $T_{L^*}$ in $U^\vee$, and piecewise linearly identifies the integral tropical points of $U$ and $U^\vee$ with the integral tropical points of $T_L$ and $T_{L^*}$ respectively, {\it{i.e.}} with the cocharacter lattices $L$ and $L^*$ of these tori.
In turn, it identifies $U^{\trop}(\Q)$ and $(U^\vee)^{\trop}(\Q)$ with a pair of dual $\Q$-vector spaces $V:=L\otimes \Q$ and $V^*:=L^*\otimes \Q$.
\begin{notation}\thlabel{not:Qvs}
For each seed $\vb{s}$, write $\mf{r}_{\vb{s}}:U^{\trop}(\Q) \to V$ and $\mf{r}_{\vb{s}}^\vee:(U^\vee)^{\trop}(\Q) \to V^*$ for the piecewise linear identifications described above.    
\end{notation}

We will also use the piecewise linear identifications $\mf{r}_{\vb{s}}$ and $\mf{r}_{\vb{s}}^\vee$ to define topologies on $U^{\trop}(\Q)$ and $(U^\vee)^{\trop}(\Q)$.
We equip both $V$ and $V^*$ with the Euclidean topology and say a set $S$ in $U^{\trop}(\Q)$ (respectively, in $(U^{\vee})^\trop(\Q)$) is open if and only if $\mf{r}_{\vb{s}}(S)$ (respectively, $\mf{r}_{\vb{s}}^\vee(S)$) is open.

We are now prepared to discuss broken lines.
Our discussion will have a few minor differences from most of the literature.
First, as in \cite{BLC}, we will need to allow broken lines to have endpoints on walls and intersect the singular locus.
Our treatment will differ only slightly from that in \cite{BLC}.
We still use a notion of {\it{generic broken lines}} and define {\it{broken lines}} as a limit of these.
However, as we are working over $\Q$, there may be regions in which all possible endpoints are contained in walls. This prevents the construction of a sequence of generic broken lines whose limit is a broken line having a prescribed endpoint in this region.
We deal with this issue by instead defining $k$-genericity with respect to the finite scattering diagram $\scat_k$ in \thref{def:k-generic}, and requiring $\gamma_k$ to be a $k$-generic broken line in \thref{def:broken}.
Next, in order to have a broken line convex geometry version of {\it{strongly convex}}, we will need a notion of a broken line which extends infinitely in both directions.
That is, broken lines are really analogous to {\it{rays}} rather than {\it{lines}}, and we need an analogue of {\it{lines}}.
For this, we introduce {\it{doubly infinite broken lines}} in \thref{def:broken2infinity}.

\begin{definition} \thlabel{def:k-generic}
Let $\scat = \colim_k \scat_k$ be a scattering diagram in $V^*$, let $m\in L^*\setminus \lrc{0}$, and let $x_0 \in V^* \setminus \supp(\scat_k)$. 
A $k$-{\it{generic broken line}} $\gamma$ with {\it{initial exponent}} $m=:I(\gamma)$ and {\it{endpoint}} $x_0$ is a piecewise linear continuous proper path $\gamma : \Q_{\leq 0} \rightarrow V^* \setminus \Sing (\scat_k)$ bending only at walls of $\scat_k$, with a finite number of domains of linearity $\ell$ and a monomial $c_\ell z^{m_\ell} \in \kk[L^*]$ for each of these domains. The path $\gamma$ and the monomials $c_\ell z^{m_\ell}$ are required to satisfy the following conditions:
\begin{itemize}
    \item $\gamma(0) = x_0$.
    \item If $\ell$ is the unique unbounded domain of linearity of $\gamma$, then $c_\ell z^{m_\ell} = z^{m}$.
    \item For $t$ in a domain of linearity $\ell$, $\dot{\gamma}(t) = -m_\ell$.
    \item Suppose $\gamma$ bends at a time $t$, passing from the domain of linearity $\ell$ to $\ell'$, and set ${\scat_t = \lrc{\left.(\wall, f_{\wall}) \in \scat \right| \gamma (t) \in \wall}}$. Then $c_{\ell'}z^{m_{\ell'}}$ is a term in $\mathfrak{p}_{{\gamma}|_{(t-\epsilon,t+\epsilon)},\scat_t} (c_\ell z^{m_\ell}) $.
\end{itemize}
\end{definition}

\begin{definition}\thlabel{def:broken}
Let $m \in L^* \setminus \lrc{0}$ and $x_0 \in V^*$. 
A {\it{broken line}} with initial exponent vector $m$ and endpoint $x_0$ is a piecewise linear continuous proper path  $\gamma: \Q_{\leq 0} \to V^*$, together with a sequence $(\gamma_k)_{k \in \Z_{>0}}$ satisfying:
\begin{itemize}
    \item $\gamma_k$ is a $k$-generic broken line;
    \item $\lrp{\supp(\gamma_k)}_{k\in \Z_{>0}}$ converges to $\text{Im}(\gamma)$;
    \item $I(\gamma_k) = m$ for all $k \in \Z_{>0}$; and
    \item for some sufficiently large $K$, all $\gamma_k$ with $k>K$ bend at the same collection of walls in the same order and have the same decorating monomials.
\end{itemize}   
We call $\text{Im}(\gamma)$ the \emph{support} of the broken line and denote it by $\supp(\gamma)$. 
\end{definition}

We can modify Definitions~\ref{def:k-generic}~and~\ref{def:broken} slightly to obtain our analogue of lines.

\begin{definition} \thlabel{def:k-generic-2infinity}
Let $\scat = \colim_k \scat_k$ be a scattering diagram in $V^*$ and let $m_1, m_2\in L^*\setminus \lrc{0}$. 
A $k$-{\it{generic doubly infinite broken line}} with {\it{initial exponent}} $m_1$ and {\it{final exponent}} $m_2$ is a piecewise linear continuous proper path $\gamma : \Q \rightarrow V^* \setminus \Sing (\scat_k)$ bending only at walls of $\scat_k$, with a finite number of domains of linearity $\ell$ and a monomial $c_\ell z^{m_\ell} \in \kk[L^*]$ for each of these domains. The path $\gamma$ and the monomials $c_\ell z^{m_\ell}$ are required to satisfy the following conditions:
\begin{itemize}
    \item $\displaystyle{\lim_{t \to -\infty}\dot{\gamma}(t) = -m_1}$ and $\displaystyle{\lim_{t \to \infty}\dot{\gamma}(t) = -m_2}$.
    \item If $\ell$ is the unbounded domain of linearity of $\gamma$ associated to times $t \ll 0$, then $c_\ell z^{m_\ell} = z^{m_1}$.
    \item For $t$ in a domain of linearity $\ell$, $\dot{\gamma}(t) = -m_\ell$.
    \item Suppose $\gamma$ bends at a time $t$, passing from the domain of linearity $\ell$ to $\ell'$, and set ${\scat_t = \lrc{\left.(\wall, f_{\wall}) \in \scat \right| \gamma (t) \in \wall}}$. Then $c_{\ell'}z^{m_{\ell'}}$ is a term in $\mathfrak{p}_{{\gamma}|_{(t-\epsilon,t+\epsilon)},\scat_t} (c_\ell z^{m_\ell}) $.
\end{itemize}
\end{definition}

\begin{definition}\thlabel{def:broken2infinity}
Let $\scat = \colim_k \scat_k$ be a scattering diagram in $V^*$ and let $m_1, m_2\in L^*\setminus \lrc{0}$. 
A {\it{doubly infinite broken line}} with {\it{initial exponent}} $m_1$ and {\it{final exponent}} $m_2$ is a piecewise linear continuous proper path $\gamma: \Q \to V^*$, together with a sequence $(\gamma_k)_{k \in \Z_{>0}}$ satisfying:
\begin{itemize}
    \item $\gamma_k$ is a $k$-generic doubly infinite broken line;
    \item $\lrp{\supp(\gamma_k)}_{k\in \Z_{>0}}$ converges to $\text{Im}(\gamma)$;
    \item $I(\gamma_k) = m_1$ for all $k \in \Z_{>0}$; and
    \item for some sufficiently large $K$, all $\gamma_k$ with $k>K$ bend at the same collection of walls in the same order and have the same decorating monomials and, in particular, have final exponent $m_2$.
\end{itemize}   
We call $\text{Im}(\gamma)$ the \emph{support} of the doubly infinite broken line and denote it by $\supp(\gamma)$. 
\end{definition}

\section{Tropical Minkowski sum}\label{sec:TMS}

In order to generalize many convex polyhedral geometry constructions of the toric world to the setting of cluster varieties,
we will need a convex tropical geometry version of the Minkowski sum.
In this section we provide such a notion and illustrate some of its key properties, particularly \thref{thm:Compatible} which illustrates the compatibility of this {\it{tropical Minkowski sum}} with the {\it{broken line convex hull}}. 
In essence, the tropical Minkowski sum of two subsets $S$ and $T$ of a tropical space $U^{\trop}(\Q)$ works the same way as the usual Minkowski sum of subsets of a Euclidean space-- we ``add'' pairs of elements $(s,t)$ with $s\in S$, $t\in T$.
However, in this setting where we have only a {\emph{piecewise}} linear structure, our ``addition'' is multivalued. 
The values that arise correspond to non-zero summands of products of $\tf$-functions.
Namely, if for some $a\in \Z_{>0}$, the function $\tf_{a\, x}$ is a non-zero summand of $\tf_{a\, s} \tf_{a\, t}$, then $x$ is a value of the ``sum'' of $s$ and $t$.

\begin{definition}\thlabel{def:TropicalMinkowskiSum}
Let $S$ and $T$ be subsets of $U^\trop(\Q)$.
We define the {\it{tropical Minkowski sum}} of $S$ and $T$ as follows:
\eqn{ S+_\tf T  :=& \lrc{x \in U^\trop(\Q) :  \exists  s \in S, t\in T, a \in \Z_{> 0} \text{ with } as, at, a x \in U^\trop(\Z) \text{ such that } \alpha_{as,at}^{a x} \neq 0 }\\
=& \lrc{
     x \in U^\trop(\Q) :  \exists  s \in S, t\in T,
     \gamma:[0,\tau] \to U^{\trop}(\Q)  \text{ with } \gamma(0) = s, \gamma(\tau)=t, \gamma(\tau/2) = x/2
}}
where $\gamma$ is a broken line segment.
\end{definition}

The equivalence of the two descriptions in \thref{def:TropicalMinkowskiSum} follows immediately from the proof of  \cite[Theorem~6.1]{BLC}.
See Figure~\ref{fig:MinkowskiSumOfPoints} for a simple example of the tropical Minkowski sum, highlighting the multivalued nature of the sum.

\noindent
\begin{center}
\begin{minipage}{.85\linewidth}
\captionsetup{type=figure}
\begin{center}
\begin{tikzpicture}[scale=.85]
    \def\x{1}
    \def\d{2}
    \def\l{3}
    \def\op{0.3}

    \path (-\l,0) coordinate (3) --++ (\l,0) coordinate (0) --++ (\l,0) coordinate (1);
    \path (0,\l) coordinate (2) --++ (0,-2*\l) coordinate (4) --++ (\l,0) coordinate (5);
    \path (5) --++ (0,2*\l) coordinate (tr) --++ (-2*\l,0) coordinate (tl) --++ (0,-2*\l) coordinate (bl);

    \draw[thick, ->, opacity=\op] (3) -- (1);
    \draw[thick, ->, opacity=\op] (2) -- (4);
    \draw[thick, ->, opacity=\op] (0) -- (5) node [pos=.5, sloped, above, opacity=\op] {$1+z^{e_2^*-e_1^*}$};

    \node[opacity=\op] at (.8,\l) {$1+z^{e_2^*}$};
    \node[opacity=\op] at (-2.1,-.35) {$1+z^{-e_1^*}$};

    \node [circle, fill, more-blue, inner sep =1.5pt] (S) at (-\d,0) {};
    \path (S) --++ (0,.35) node [color=more-blue] {$S$};
    \node [circle, fill, more-green, inner sep =1.5pt] (T) at (\d,0) {};
    \path (T) --++ (0,.35) node [color=more-green] {$T$};

    \coordinate (bend) at (0,.5*\d) {};

    \draw [very thick, color=violet] (S) -- (T);
    \draw [very thick, color=violet] (S) -- (bend) -- (T);
    \node [very thick, color=violet] at (0.03,.25*\d) {$\boldsymbol{\vdots}$};

\begin{scope}[xshift=8.5cm]

    \path (-\l,0) coordinate (3) --++ (\l,0) coordinate (0) --++ (\l,0) coordinate (1);
    \path (0,\l) coordinate (2) --++ (0,-2*\l) coordinate (4) --++ (\l,0) coordinate (5);
    \path (5) --++ (0,2*\l) coordinate (tr) --++ (-2*\l,0) coordinate (tl) --++ (0,-2*\l) coordinate (bl);

    \draw[thick, ->, opacity=\op] (3) -- (1);
    \draw[thick, ->, opacity=\op] (2) -- (4);
    \draw[thick, ->, opacity=\op] (0) -- (5) node [pos=.5, sloped, above, opacity=\op] {$1+z^{e_2^*-e_1^*}$};

    \node[opacity=\op] at (.8,\l) {$1+z^{e_2^*}$};
    \node[opacity=\op] at (-2.1,-.35) {$1+z^{-e_1^*}$};

    \node [circle, fill, more-blue, inner sep =1.5pt] (S) at (-\d,0) {};
    \path (S) --++ (0,.35) node [color=more-blue] {$S$};
    \node [circle, fill, more-green, inner sep =1.5pt] (T) at (\d,0) {};
    \path (T) --++ (0,.35) node [color=more-green] {$T$};

    \coordinate (bend) at (0,.5*\d) {};
    
\draw [ultra thick, color=violet] (0,\d) -- (0);
    \node [color=violet] at (0,\d+.35) {$S\+t T$};

\end{scope}

\end{tikzpicture}
~                      

\captionof{figure}{\label{fig:MinkowskiSumOfPoints} The tropical Minkowski sum of two points in $ (\cA^\vee)^{\trop}(\Q)$ for the $\cA$ cluster variety of type $A_2$.  
As is standard, to draw this picture we identify $(\cA^\vee)^{\trop}(\Q)$ with $\Q^2$ via a choice of seed.
The relevant broken lines appear on the left and the corresponding tropical Minkowski sum on the right.
} 
\end{center}
\end{minipage}
\end{center}

\begin{remark}\thlabel{rem:LinInd} 
Consider a function $f\in \ssO(U^\vee)$ given as linear combination of products of theta functions $f=\sum_{s,t\in U^{\trop}(\Z)} c_{s,t} \tf_{s}\cdot \tf_{t}$.
Since $f\in \ssO(U^\vee)$, we may also expand it as $f= \sum_{x \in U^{\trop}(\Z)} f_x \tf_x$. 
Let $\tf_{x_0}$ be one such non-zero summand of $f$. 
Then, there exist $s_0,t_0\in U^{\trop}(\Z)$ with $c_{s_0,t_0} \neq 0$ such that $\tf_{x_0}$ is a non-zero summand of $\tf_{s_0}\cdot \tf_{t_0}$.
To see this, note that we have
\eqn{f &= \sum_{s,t\in U^{\trop}(\Z)} c_{s,t} \tf_{s}\cdot \tf_{t}\\
&= \sum_{s,t\in U^{\trop}(\Z)} \sum_{x \in U^{\trop}(\Z)}c_{s,t} \alpha_{s,t}^{x}\tf_{x}\\
&= \sum_{x \in U^{\trop}(\Z)} f_x \tf_x.
}
By linear independence of theta functions, we must have that $f_x = \sum_{s,t\in U^{\trop}(\Z)} c_{s,t} \alpha_{s,t}^{x}$ for each $x\in U^{\trop}(\Z)$. 
So $f_{x_0}$ may only be non-zero if we have some $s_0$, $t_0$ with $c_{s_0,t_0}$ and $\alpha_{s_0,t_0}^{x_0}$ both non-zero.
Note that this argument also applies if we replace the products of pairs of theta functions with products of arbitrary finite numbers of theta functions.
\end{remark}

\begin{remark} \thlabel{rem:StrongPositivity}
    The non-negativity of scattering functions for cluster scattering diagrams implies that all structure constants $\alpha_{p,q}^{r}$  (or more generally $\alpha_{p_1, \dots, p_d}^r$) are non-negative.
    This result is sometimes referred to as {\it{strong positivity}}.  Versions of this result are due to \cite[Theorem~7.5]{GHKK}, \cite[Proposition~2.15]{Travis}, and \cite[Theorem~1.1] {DavisonMandel}.
\end{remark}

\begin{lemma}\thlabel{lem:rescale}
    Let $p, q, r \in U^{\trop}(\Z)$ be such that $\alpha_{p,q}^{r} \neq 0$, and let $a\in \Z_{>0}$.  Then $\alpha_{a p, a q}^{a r} \neq 0$.
\end{lemma}

\begin{proof}
    If $(\gamma_1,\gamma_2)$ is a pair of broken lines contributing to $\alpha_{p,q}^{r}$,
    we may rescale the exponent vectors of decoration monomials as well as the supports of $\gamma_1$ and $\gamma_2$ by a factor of $a$ to obtain a new pair broken lines (of higher multiplicity) $\lrp{\widetilde{\gamma}_1,\widetilde{\gamma}_2}$ contributing to $\alpha_{a p, a q}^{a r}$.
    Then positivity of scattering functions implies no cancellations may occur and $\alpha_{a p, a q}^{a r} \neq 0$.
\end{proof}

\begin{lemma}\thlabel{lem:axsumaix}
    Let $x \in U^{\trop}(\Q)$, and let $a_1,\dots, a_d $ be non-negative integers such that each $a_i x$ is integral. Then $\alpha_{a_1 x,\dots, a_d x}^{(a_1+\cdots + a_d)x} \neq 0$.
\end{lemma}

\begin{proof}
Choose a seed $\vb{s}$ to identify $U^{\trop}(\Q)$ with a $\Q$-vector space $V$ by a map $\mf{r}_{\vb{s}}$ as in \thref{not:Qvs}.
Take $(\gamma_1,\cdots,\gamma_d)$ to be the collection of straight broken lines in $V$ where the initial decoration monomial of $\gamma_i$ is $z^{a_i \mf{r}_{\vb{s}}(x)}$ and the endpoint of each $\gamma_i$ is $(a_1+\cdots + a_d)\mf{r}_{\vb{s}}(x)$.
This contributes $1$ to  $\alpha_{a_1 x,\dots, a_d x}^{(a_1+\cdots + a_d)x}$. (It is in fact the only contribution.)
\end{proof}

\begin{lemma}\thlabel{lem:TMSAssociative}
Let $S$, $T$ and $R$ be subsets of $U^\trop(\Q)$. Then
\eqn{ (S +_\tf T) +_\tf R
= S +_\tf (T +_\tf R). } 
\end{lemma}
\begin{proof}
Let $x\in (S +_\tf T) +_\tf R$. 
Then there exists $y\in S +_\tf T$, $r\in R$ and $a\in\Z_{>0}$ with $\alpha_{ay,ar}^{ax}\neq 0$, 
meaning $\tf_{ax}$ is a non-zero summand of $\tf_{ay}\cdot\tf_{ar}$. 
Similarly, since $y\in S +_\tf T$ we have $\alpha_{bs,bt}^{by}\neq 0$ for some $s\in S$, $t\in T$ and $b\in\Z_{>0}$, meaning $\tf_{by}$ is a non-zero summand of $\tf_{bs}\cdot\tf_{bt}$. 
Then, since 
\eqn{\tf_{abx} \text{ is a non-zero summand of }\tf_{aby}\cdot\tf_{abr}}
and
\eqn{\tf_{aby} \text{ is a non-zero summand of }\tf_{abs}\cdot\tf_{abt},}
we obtain that 
\eqn{\tf_{abx} \text{ is a non-zero summand of }\tf_{abs}\cdot\tf_{abt}\cdot\tf_{abr}.}
Consider the expression $\tf_{abt}\cdot\tf_{abr}= \sum_{abz\in U^\trop(\Z)} \alpha_{abt,abr}^{abz} \tf_{abz}$.
By construction, if $\alpha_{abt,abr}^{abz}\neq 0$, then $z \in T+_\tf R$.
We have now that 
\eqn{\tf_{abx} \text{ is a non-zero summand of }\sum_{abz\in U^\trop(\Z)} \alpha_{abt,abr}^{abz} \tf_{abs}\cdot\tf_{abz}.}
Then by \thref{rem:LinInd}, we find that $\tf_{abx}$ is a non-zero summand of $\tf_{abs}\cdot\tf_{abz}$ for some $z \in T+_\tf R$. Consequently, $x\in S+_\tf (T+_\tf R)$.
\end{proof}

\begin{prop}\thlabel{char_blc}
A subset $S$ of $U^{\trop}(\Q)$ is broken line convex if and only if for all 
$t\in [0,1]$, we have
\eqn{
t S +_\tf (1-t) S = S.
}
\end{prop}

\begin{proof}
Let $S$ be broken line convex.
Then for all $a$, $b$ in $\Z_{\geq 0}$, $p\in a S(\Z)$, $q \in b S(\Z)$ and $r\in U^{\trop}(\Z)$ with $\alpha_{p,q}^r \neq 0$, we have that $r\in (a+b) S$.
If $z \in t S +_\tf (1-t) S$, then there is some $x\in t S$, $y\in (1-t) S$, and $c\in \Z_{>0}$ such that $c\, x$, $c\, y $, and $c\, z$ are in $U^{\trop}(\Z)$ and $\alpha_{c\, x, c\, y}^{c\, z} \neq 0$.
We can find non-negative integers $a$ and $b$ such that $t=\frac{a}{a+b}$ and $c = a+b$.
Then $p:= c\, x \in a S (\Z)$, $q:= c\, y \in b S (\Z)$, and $r:= c\, z$ must be in $(a+b) S$.
It follows that $z \in S$, and $t S +_\tf (1-t) S \subset S$.

On the other hand for all $z \in U^{\trop}(\Q)$, we can draw a straight line segment from $t z$ to $(1-t)z$.
As such, $z \in t \lrc{z} +_{\tf} (1-t)\lrc{z}$.
So if $z \in S$, then $z \in t \lrc{z} +_{\tf} (1-t)\lrc{z} \subset t S +_\tf (1-t) S$, and
$S \subset t S +_\tf (1-t) S $.

Now suppose $t S +_\tf (1-t) S = S$ for all 
$t \in [0,1]$.
We want to show that for all $a$, $b$ in $\Z_{\geq 0}$, $p\in a S(\Z)$, $q \in b S(\Z)$ and $r\in U^{\trop}(\Z)$ with $\alpha_{p,q}^r \neq 0$, we have $r\in (a+b) S$.
First we address the trivial case:  if $a=b=0$ and $\alpha_{p,q}^r \neq 0$, then necessarily $p=q=r=0 \in 0\cdot S$. 
Next, assume $a>0$ or $b>0$, and let $t= \frac{a}{a+b}$.
Write $p':=\frac{p}{a+b}$, $q':=\frac{q}{a+b}$, and $r':=\frac{r}{a+b}$, so 
$p'\in t S$, $q'\in (1-t)S$, and $\alpha_{(a+b)p',(a+b)q'}^{(a+b)r'}\neq 0$.
This implies $r'\in t S +_{\tf}(1-t)S =S$, so $r\in (a+b)S$ as desired.
\end{proof}

\begin{prop}\thlabel{blctropms}
If the subsets $S$ and $T$ of $U^{\trop}(\Q)$ are broken line convex, then $S +_\tf T$ is broken line convex.
\end{prop} 
\begin{proof}
Let $\tau\in [0,1]$. 
If we prove the equality 
\eqn{\tau(S +_{\tf} T)+_{\tf} (1-\tau)(S +_{\tf} T)= S +_{\tf} T,}
then by \thref{char_blc} we conclude the result. 

Assume that $x\in \tau(S +_{\tf} T)+_{\tf} (1-\tau)(S +_{\tf} T)$. So, there exist $y\in \tau\cdot (S +_{\tf} T)$, $z\in (1-\tau)\cdot (S +_{\tf} T)$ and $a\in\Z_{>0}$ such that $\alpha_{a\, y, a\, z}^{a\, x}\neq 0$, meaning $\tf_{ax}$ is a non-zero summand of $\tf_{ay}\cdot\tf_{az}$. Now, since  $y\in \tau\cdot (S +_{\tf} T)$, there exist $s_1\in S$, $t_1\in T$ and $b\in\Z_{>0}$ such that $\alpha_{b\,\tau\, s_1, b\,\tau\, t_1}^{b\, y}\neq 0$, meaning $\tf_{by}$ is a non-zero summand of $\tf_{b\tau s_1}\cdot \tf_{b\tau t_1}$. 
Similarly, $z\in (1-\tau)\cdot (S +_{\tf} T)$ implies the existence of $s_2\in S$, $t_2\in T$ and $c\in\Z_{>0}$ such that $\alpha_{c\,(1-\tau)\, s_2, c\, (1-\tau)\, t_2}^{c\, z}\neq 0$, meaning $\tf_{cz}$ is a non-zero summand of $\tf_{c (1-\tau) s_2}\cdot \tf_{c (1-\tau) t_2}$. Then, it follows that
\eqn{\tf_{abcx} &\text{ is a non-zero summand of } \tf_{abcy}\cdot\tf_{abcz}, \\
\tf_{abcy} &\text{ is a non-zero summand of } \tf_{abc\tau s_1}\cdot\tf_{abc\tau t_1}, \\
\tf_{abcz} &\text{ is a non-zero summand of }\tf_{abc(1-\tau) s_2}\cdot\tf_{abc(1-\tau) t_2}.
}
Then, we have that
\begin{equation}\label{eq:tfabcx}
\tf_{abcx} \text{ is a non-zero summand of }
\tf_{abc\tau s_1}\cdot\tf_{abc(1-\tau) s_2}\cdot\tf_{abc\tau t_1}\cdot\tf_{abc(1-\tau) t_2}.
\end{equation}
Now, consider the expressions
\eqn{\tf_{abc\tau s_1}\cdot\tf_{abc(1-\tau) s_2}&= \sum_{abcs\in U^\trop(\Z)} \alpha_{abc\tau s_1,abc(1-\tau) s_2}^{abcs} \tf_{abcs}, \quad \text{and} \\
\tf_{abc\tau t_1}\cdot\tf_{abc(1-\tau) t_2}&= \sum_{abct\in U^\trop(\Z)} \alpha_{abc\tau t_1,abc(1-\tau) t_2}^{abcs} \tf_{abct}.
}
By Equation~\eqref{eq:tfabcx} and \thref{rem:LinInd} we have that $\tf_{abcx}$ is a non-zero summand of $\tf_{abcs} \cdot \tf_{abct}$ for some $s\in S$ and $t\in T$. Therefore, we have that $x\in S +_{\tf} T$ and we conclude that $\tau (S +_{\tf} T) +_{\tf} (1-\tau)(S +_{\tf} T)\subseteq S +_{\tf} T$.

For the other containment, consider $x\in S +_\tf T$. If we consider the line segment $\tau x + (1-\tau)x$, then we have that $x\in \tau\{x\} +_{\tf} (1-\tau)\{x\}$. So, since $\tau\{x\} +_{\tf} (1-\tau)\{x\}\subset \tau (S +_{\tf} T) +_{\tf} (1-\tau)(S +_{\tf} T)$ we conclude that $S +_{\tf} T\subseteq \tau (S +_{\tf} T) +_{\tf} (1-\tau)(S +_{\tf} T)$. 
\end{proof}

\begin{cor}\thlabel{cor:ConvSumInSumConv}
Let $S$ and $T$ be subsets of $U^{\trop}(\Q)$. Then
\eqn{\bconv(S +_{\tf} T) \subseteq \bconv(S) +_{\tf} \bconv(T).}
\end{cor}
\begin{proof}
Since $S +_{\tf} T \subseteq \bconv(S) +_{\tf} \bconv(T)$, we have that 
\eqn{
\bconv(S +_{\tf} T) \subseteq \bconv(\bconv(S) +_{\tf} \bconv(T)).
}
Note that $\bconv(S)$ and $\bconv(T)$ are broken line convex sets, then \thref{blctropms} implies that $\bconv(S) +_{\tf} \bconv(T)$ is a broken line convex set and consequently we obtain that $\bconv(\bconv(S) +_{\tf} \bconv(T))=\bconv(S) +_{\tf} \bconv(T)$. The claim follows.
\end{proof}

\begin{definition}\thlabel{def:Sd}
Let $\alpha_{p_1,\dots,p_d}^r$ denote the coefficient of $\tf_r$ in the expansion of $\tf_{p_1}\cdots \tf_{p_d}$.
For $S \subset U^\trop(\Q)$ define
\eqn{S_d := \lrc{u \in U^\trop(\Q): \alpha_{a_1\, s_1, \dots, a_d\, s_d}^{(a_1+\cdots+ a_d)u}\neq 0 \text { for some } s_1,\dots, s_d \in S \text { and } a_1,\dots, a_d \in \Z_{\geq 0}, \text{ with }  \sum_{i=1}^{d}a_i \neq 0}. }  
\end{definition}

\begin{lemma}\thlabel{lem:filtration}
We have a filtration
$S= S_1 \subset S_2 \subset \cdots $.
\end{lemma}

\begin{proof}
The first equality is immediate from the definition of $S_1$. 
For the remaining containments, set $a_{d+1}=0$ to find $S_d \subset S_{d+1}$.
\end{proof}

\begin{lemma}\thlabel{lem:SdPartialColim}
If $x\in S_{d_1}$, $y\in S_{d_2}$, and $\alpha_{n\, x, m\, y}^{(n+m)z} \neq 0$, then $z\in S_{d_1+d_2}$.
\end{lemma}

\begin{proof}
First, since $x\in S_{d_1}$, we have $\alpha_{a_1\, s_1, \dots, a_{d_1}\, s_{d_1}}^{{\bf{a}}\, x} \neq 0$ for some $s_1,\dots, s_{d_1} \in S$ and  $a_1,\dots, a_{d_1} \in \Z_{\geq 0}$  with  ${\bf{a}}:=\sum_{i=1}^{d_1}a_i \neq 0$, 
and 
\eq{\tf_{{\bf{a}}\, x} \text{ is a non-zero summand of }\tf_{a_1\, s_1}\cdots \tf_{a_{d_1}\, s_{d_1}}.}{eq:ax}
Similarly, since $y\in S_{d_2}$, we have $\alpha_{b_1\, r_1, \dots, b_{d_2}\, r_{d_2}}^{{\bf{b}}\, y} \neq 0$ for some $r_1,\dots, r_{d_2} \in S$ and  $b_1,\dots, b_{d_2} \in \Z_{\geq 0}$  with  ${\bf{b}}:=\sum_{i=1}^{d_2}b_i \neq 0$, 
and 
\eq{\tf_{{\bf{b}}\, y} \text{ is a non-zero summand of }\tf_{b_1\, r_1}\cdots \tf_{b_{d_2}\, r_{d_2}}.}{eq:by}
Next, since $\alpha_{n\, x, m\, y}^{(n+m)z} \neq 0$, 
\eq{\tf_{(n+m)z} \text{ is a non-zero summand of }\tf_{n\, x} \cdot \tf_{m\, y}.}{eq:xyz}

We claim that $\tf_{\mathbf{a}\, \mathbf{b}\,(n+m)z}$ is a non-zero summand of \eqn{\tf_{n\, \mathbf{b}\, a_1\, s_1}\cdots \tf_{n\, \mathbf{b}\, a_{d_1}\, s_{d_1}}\cdot \tf_{m\, \mathbf{a}\, b_1\, r_1}\cdots \tf_{m\, \mathbf{a}\, b_{d_2}\, r_{d_2}}.}
First, using \thref{lem:rescale}, we can conclude from \eqref{eq:ax}
that
\eq{\tf_{n\, {\bf a\, b}\, x} \text{ is a non-zero summand of }\tf_{n\, {\bf b}\, a_1\, s_1}\cdots \tf_{n\, {\bf b}\, a_{d_1}\, s_{d_1}}.}{eq:nabx}
The same argument applied to \eqref{eq:by} shows 
\eq{\tf_{m\,\mathbf{a}\,{\bf b}\,y} \text{ is a non-zero summand of }\tf_{m\,\mathbf{a}\, b_1\, r_1}\cdots \tf_{m\,\mathbf{a}\, b_{d_2}\, r_{d_2}}.}{eq:maby}

Next, since the structure constants are non-negative, by \thref{rem:StrongPositivity},
\eqref{eq:nabx}, and \eqref{eq:maby}, 
\eq{&\text{non-zero summands of } \tf_{n\, {\bf a\, b}\, x} \cdot \tf_{m\,\mathbf{a}\,{\bf b}\,y} \text{ must also be}\\
&\text{non-zero summands of } \tf_{n\, {\bf b}\, a_1\, s_1}\cdots \tf_{n\, {\bf b}\, a_{d_1}\, s_{d_1}} \cdot \tf_{m\,\mathbf{a}\, b_1\, r_1}\cdots \tf_{m\,\mathbf{a}\, b_{d_2}\, r_{d_2}}.}{eq:nabxmaby}
Finally, we can conclude from \thref{lem:rescale} and \eqref{eq:xyz} that 
\eq{\tf_{(n+m){\bf a\, b}\, z} \text{ is a non-zero summand of }\tf_{n\,{\bf a\, b}\,  x} \cdot \tf_{m\,{\bf a\, b}\,  y}.}{eq:nmabxyz}
Combining \eqref{eq:nabxmaby} and \eqref{eq:nmabxyz} finishes the proof.
\end{proof}

\begin{lemma}\thlabel{lem:MultiProdPositivity}
$S$ is positive if and only if for any $n>0$, $a_1, \dots, a_n \in \Z_{\geq 0}$, $s_i \in a_i S(\Z)$, and $r \in U^{\trop}(\Z)$ with $\alpha_{s_1, \dots, s_n}^{r} \neq 0$, we have 
$r \in (a_1 + \dots + a_n) S$.
\end{lemma}

\begin{proof}
For $n=2$, this is the definition of positivity, so the if part holds.
Next, if $S$ is positive, we use associativity of theta function multiplication to conclude the only if part. 
\end{proof}

\begin{lemma}\thlabel{lem:SdInConvS}
For all $d\in \Z_{>0}$, we have 
$\displaystyle{S_d \subset \bconv(S)}$.
\end{lemma}

\begin{proof}
If $u\in S_d$, we have $\alpha_{a_1\, s_1, \dots, a_d\, s_d}^{(a_1+\cdots a_d)u} \neq 0$ for some  $s_1,\dots, s_d \in S$ and $a_1,\dots, a_d \in \Z_{\geq 0}$  with $\sum_{i=1}^{d}a_i \neq 0$.
As $S  \subset \bconv(S)$, and positivity is equivalent to broken line convexity, $(a_1+\cdots +a_d)u$ must be in $(a_1+\cdots +a_d) \bconv(S)$.  By \thref{lem:MultiProdPositivity}, failure of this would contradict positivity of $\bconv(S)$.  So $u \subset \bconv(S)$, and $S_d \subset \bconv(S)$.
\end{proof}

\begin{cor}\thlabel{cor:ConvColim}
Let $S$ be any subset of $U^{\trop}(\Q)$.  Then
\eqn{\bconv(S)= \bigcup_{d\geq 1} S_d.}
\end{cor}

\begin{proof}
By \thref{lem:SdPartialColim}, the infinite union $\bigcup_{d\geq 1} S_d$ is positive, and hence broken line convex.
As it is broken line convex and contains $S$ (see \thref{lem:filtration}), we find that 
\eqn{\bconv(S) \subset \bigcup_{d\geq 1} S_d.}
By \thref{lem:SdInConvS}, we observe the opposite inclusion:
\eqn{\bconv(S) \supset \bigcup_{d\geq 1} S_d.}
\end{proof}

\begin{lemma}\thlabel{lem:Sd+Te}
Let $S$ and $T$ be subsets of $U^{\trop}(\Q)$.
For all $d,e\in\mathbb{Z}_{>0}$,
\eqn{S_d +_{\tf} T_e \subset 
(S +_{\tf} T)_{de}. }
\end{lemma}

\begin{proof}
Let $x$ be in the sum $S_d +_{\tf} T_e$.
Then there is some $s\in S_d$, $t\in T_e$, and $a\in \Z_{>0}$ such that $as$, $at$, and $ax$ are all integral and $\alpha_{as, at}^{ax}\neq 0$. 
That is, 
\eq{\tf_{ax} \text{ is a non-zero summand of } \tf_{as}\cdot\tf_{at}.}{fact:axst} 
Now, since $s\in S_d$, there exist $s_1,\dots,s_d\in S$ and $b_1,\dots,b_d\in\Z_{\geq 0}$ such that $\alpha_{b_1\,s_1,\dots,b_d\,s_d}^{\vb{b}\, s}\neq 0$, where $\vb{b}:=b_1+\cdots+b_d$. 
That is, 
\eq{\tf_{\vb{b}\, s} \text{ is a non-zero summand of } \tf_{b_1 s_1}\cdots\tf_{b_d s_d}.}{fact:bsd} 
If any of these integers $b_i$ is $0$, we may simply replace $d$ by a smaller $d'$ using \thref{lem:filtration}.
So, we may assume $b_i>0$ for all $i\in \lrc{1,\dots,d}$.
Similarly, since $t\in T_e$, there exist $t_1,\dots,t_e\in T$ and $c_1,\dots,c_e\in\Z_{\geq 0}$ such that $\alpha_{c_1\,t_1,\dots,c_e\,t_e}^{\vb{c}\, t}\neq 0$, where $\vb{c}:=c_1+\cdots+c_e$. 
That is, 
\eq{\tf_{\vb{c}\, t} \text{ is a non-zero summand of } \tf_{c_1t_1}\cdots\tf_{c_e t_e}.}{fact:cte}
As before, we may assume $c_j>0$ for all $j\in \lrc{1,\dots,e}$.

Rescaling coefficients using \thref{lem:rescale}, the facts
\eqref{fact:axst}, \eqref{fact:bsd}, and \eqref{fact:cte} imply
\eq{\tf_{a\, \vb{b}\, \vb{c}\, x} \text{ is a non-zero summand of } \tf_{a\, \vb{b}\, \vb{c}\, s}\cdot\tf_{a\, \vb{b}\, \vb{c}\, t},}{fact:abcxst} 
\eq{\tf_{a\, \vb{b}\, \vb{c}\, s} \text{ is a non-zero summand of } \tf_{a\, b_1\, \vb{c}\,  s_1}\cdots\tf_{a\, b_d\, \vb{c}\,  s_d},}{fact:abcsd} 
and
\eq{\tf_{a\, \vb{b}\, \vb{c}\, t} \text{ is a non-zero summand of } \tf_{a\, \vb{b}\, c_1\, t_1}\cdots\tf_{a\, \vb{b}\, c_e\,  t_e}}{fact:abcte}
respectively.
Moreover, using \thref{lem:axsumaix}, we have that 
\eq{\tf_{a\, b_i\, \vb{c}\,  s_i} \text{ is a non-zero summand of } \tf_{a\, b_i\, c_1\,  s_i}\cdots \tf_{a\, b_i\, c_e\,  s_i} }{fact:abicsi}
and
\eq{\tf_{a\, \vb{b}\, c_j\,  t_j} \text{ is a non-zero summand of } \tf_{a\, b_1\, c_j\,  t_j}\cdots \tf_{a\, b_d\, c_j\,  t_j} .}{fact:abcjstj}
Next, using Remarks~\ref{rem:LinInd} and \ref{rem:StrongPositivity}, the facts \eqref{fact:abcxst}, \eqref{fact:abcsd}, \eqref{fact:abcte}, \eqref{fact:abicsi}, and \eqref{fact:abcjstj} together imply
\eq{\tf_{a\, \vb{b}\, \vb{c}\, x} \text{ is a non-zero summand of } \prod_{\substack{i\in \lrc{1,\dots,d}\\ j\in \lrc{1,\dots,e}}} \tf_{a b_i c_j s_i} \cdot \tf_{a b_i c_j t_j}.}{fact:prodij}
Expanding each product $\tf_{a b_i c_j s_i} \cdot \tf_{a b_i c_j t_j}$ and using \thref{rem:LinInd} once more, we find that 
\eq{\tf_{a\, \vb{b}\, \vb{c}\, x} \text{ is a non-zero summand of } \prod_{\substack{i\in \lrc{1,\dots,d}\\ j\in \lrc{1,\dots,e}}} \tf_{a b_i c_j r_{ij}}}{fact:prodrij}
for some collections of elements $\lrc{r_{ij}\in S+_\tf T: i\in \lrc{1,\dots,d},\, j\in \lrc{1,\dots,e}}$.
Finally, observe that 
\eq{\sum_{\substack{i\in \lrc{1,\dots,d}\\ j\in \lrc{1,\dots,e}}} a\, b_i\, c_j = a\, \vb{b}\, \vb{c}.}{eq:abc}
Thus, \eqref{fact:prodrij} and \eqref{eq:abc} imply $x \in \lrp{S+_\tf T}_{d e}$, as claimed.
\end{proof}

\begin{cor}\thlabel{cor:SumConvInConvSum}
Let $S$ and $T$ be subsets of $U^{\trop}(\Q)$.  Then
 \eqn{\bconv(S) +_\tf  \bconv(T) \subset \bconv(S+_\tf T).}
\end{cor}

\begin{proof}
This is an immediate consequence of \thref{cor:ConvColim} and  \thref{lem:Sd+Te}.     
\end{proof}

Combining \thref{cor:ConvSumInSumConv} and \thref{cor:SumConvInConvSum}, we obtain that the tropical Minkowski sum and broken line convex hull are compatible in the following sense:

\begin{theorem}\thlabel{thm:Compatible}
Let $S$ and $T$ be subsets of $U^{\trop}(\Q)$.  Then
\eqn{\bconv(S+_\tf T) = \bconv(S)+_\tf \bconv(T).}
\end{theorem}

To conclude this section, we provide another result relating the tropical Minkowski sum and broken line convex hull.
It will come in handy in later sections.

\begin{prop}\thlabel{prop:ConvUnion-UnionSum}
	Let $S = \bigcup_{i\in I} S^i$, where each $S^i \subset U^{\trop}(\Q)$ is broken line convex.
	Then 
	\eq{\bconv(S) = \bigcup_{\substack{(a_i:i\in I) \in (\Q_{\geq 0})^{I}\\ \sum_{i\in I}a_i = 1}} \lrp{ \Sumt_{i\in I}  a_i S^i}. }{eq:MinkSumOfBLCComps}
\end{prop}

\begin{proof}
    First, let $(a_i:i\in I) \in (\Q_{\geq 0})^{I}$ with $\sum_{i\in I}a_i = 1$, and let $s\in \displaystyle{\Sumt_{i\in I} } a_i S^i$.
    Each $a_i S^i$ is broken line convex since each $S^i$ is.
    If $\lrm{I}=1$, there is nothing to show.
    Next let $I=\lrc{1,2}$. 
    Then there is some $x_1 \in a_1 S^1$, $x_2 \in a_2 S^2$, and $c \in \Z_{>0}$ such that $c\, x_1$, $c\, x_2$, and $c\, s$ are all integral and $\alpha_{c\, x_1,c\, x_2}^{c\, s} \neq 0$.
    The case in which either $a_i$ is zero reduces to the $\lrm{I}=1$ case, so we may assume each $a_i$ is non-zero.
    Write $a_i = \frac{n_i}{d_i}$, with $n_i, d_i\in \Z_{> 0}$.
    Then $\alpha_{c d_1 d_2 x_1, c d_1 d_2 x_2}^{c d_1 d_2 s} \neq 0$ as well by \thref{lem:rescale}.
    But $c d_1 d_2 x_1 \in c d_2 n_1 S^1$ and $c d_1 d_2 x_2 \in c d_1 n_2 S^2$.
    So \cite[Proposition~4.10, Theorem~6.1]{BLC} implies there is a broken line segment from $\frac{d_1}{n_1} x_1 = {a_1}^{-1} x_1$ to $\frac{d_2}{n_2} x_2 = {a_2}^{-1}x_2$ passing through $\frac{d_1 d_2}{d_2 n_1 + d_1 n_2} s= (a_1+a_2)^{-1} s = s$.
    Since ${a_i}^{-1}x_i \in S^i\subset S$, this implies $s\in \bconv(S)$.
    Now suppose the right side of \eqref{eq:MinkSumOfBLCComps} is contained in the left whenever $\lrm{I}<r$, and consider the case $I=\lrc{1,\dots,r}$.
    If any $a_i = 0$, we return to the $\lrm{I}<r$ case.
    So assume each $a_i$ is non-zero.
    Let $\vb{a}= a_1 + \cdots +a_{r-1}$, and let $a_i'= \frac{a_i}{\vb{a}}$ for $i\in  I\setminus\lrc{r}=:I'$.
    By the induction hypothesis, we know that 
    \eqn{\Sumt_{i\in I'}a_i' S^i \subset \bconv\lrp{\bigcup_{i\in I'}S^i}=:S'.}
    So, $s\in \displaystyle{\Sumt_{i\in I} } a_i S^i \subset \vb{a}S' \+t a_r S^r $.
    But by the induction hypothesis, $\vb{a}S' \+t a_r S^r  \subset \bconv(S' \cup S^r) = \bconv(S)$.
    So
    \eqn{
    \bigcup_{\substack{(a_i:i\in I) \in (\Q_{\geq 0})^{I}\\ \sum_{i\in I}a_i = 1}} \lrp{ \Sumt_{i\in I}  a_i S^i} \subset \bconv(S). }

    Now suppose $s \in \bconv(S)$.
    By \thref{cor:ConvColim}, $s\in S_d$ (from \thref{def:Sd}) for some $d \in \Z_{>0}$.
    So, we can find some $s_1, \dots, s_d \in S$ and $a_1, \dots, a_d \in \Z_{\geq 0}$ with $a_1+\cdots +a_d \neq 0$, the tropical points $a_1\, s_1, \dots, a_d\, s_d$, and $(a_1+\cdots +a_d)s$ all integral, and the structure constant $\alpha_{a_1\, s_1, \dots, a_d\, s_d}^{(a_1+\cdots+ a_d)s}\neq 0$.
    Each $s_j$ is in some $S^i$.
    Let $\bigcup_{i \in I} J_i$ be a decomposition of $\lrc{1,\dots,d} $ as a disjoint union such that
    $j\in J_i$ only if $s_j \in S^i$.\footnote{The point here is that $s_j$ may be contained in multiple $S^i$'s.  We simply choose one such $i$.} 
    Now we have that 
    \eqn{\tf_{(a_1+\cdots +a_d)s} \text{ is a non-zero summand of } \prod_{i\in I} \lrp{ \prod_{j \in J_i} \tf_{a_{j} s_{j}}} = \prod_{i\in I} \lrp{ \sum_{x \in U^{\trop}(\Z)} \alpha_{\lrc{a_{j} s_{j}: j \in J_i }}^{x} \tf_{x} }.} 
    By \thref{rem:LinInd},
    we can find a collection $\lrc{x_i \in U^{\trop}(\Z): i\in I , \alpha_{\lrc{a_{j} s_{j}: j \in J_i }}^{x_i} \neq 0}$
    such that 
    \eqn{\tf_{(a_1+\cdots +a_d)s} \text{ is a non-zero summand of } \prod_{i\in I}  \tf_{x_{i}}.} 
    Since $S^i$ is broken line convex, $x_i \in (\sum_{j \in J_i} a_j)S^i$.
    Then 
    \eqn{(a_1+\cdots +a_d)s \in \Sumt_{i\in I}  \vb{a}_{J_i} S^i,} 
    where $  \vb{a}_{J_i}:=  \sum_{j \in J_i} a_j$, and
    \eqn{s \in \Sumt_{i\in I}  \frac{\vb{a}_{J_i}}{(a_1+\cdots +a_d)} S^i.}
    We conclude that
    \eqn{\bconv(S) \subset \bigcup_{\substack{(a_i:i\in I) \in (\Q_{\geq 0})^{I}\\ \sum_{i\in I}a_i = 1}} \lrp{ \Sumt_{i\in I}  a_i S^i} }
    as well.
\end{proof}

\section{Convexity for functions on \texorpdfstring{$U^{\trop}(\Q)$}{tropical spaces}}\label{sec:Functions}

\subsection{Definition and characterization}

In this section we describe what it means for a function on $U^{\trop}(\Q)$ to be convex, and we prove some key results about these convex functions.

\begin{definition}\thlabel{def:ConvWRTBL}
Let $S\subset U^{\trop}(\Q)$ be a broken line convex set. 
A function $\varphi:S\rightarrow \Q $ is \textit{convex with respect to broken lines} if for 
any broken line segment $\gamma:[t_1,t_2]\rightarrow S$, 
we have that    
\eq{\varphi(\gamma(t)) \geq \lrp{\dfrac{t_2-t}{t_2-t_1}} \varphi(\gamma(t_1)) + \lrp{\dfrac{t-t_1}{t_2-t_1}} \varphi(\gamma(t_2))}{eq:ConvWRTBLineq1} 
for all $t\in [t_1,t_2]$.
\end{definition}

We would like to draw attention to the direction of the inequality in \thref{def:ConvWRTBL}.
It is very common to see the opposite inequality in the definition of a convex function in the linear setting.
We have chosen our conventions to match those of \cite[Definition~6.1.4]{CLS} (where a similar word of warning is provided) and \cite[Definition-Lemma~8.1.(2)]{GHKK}.

\begin{prop}[Broken Line Jensen's Inequality]\thlabel{prop:ConvWRTBL}
    Let $S\subset U^{\trop}(\Q)$ be broken line convex, and let $\varphi:S \to \Q$ be convex with respect to broken lines.
    If $s_1,\dots, s_d, s \in S$, $a_1, \dots, a_d \in \Q_{\geq 0}$
    with $a_1 s_1,\dots, a_d s_d$, and $(a_1+\cdots +a_d)s$ all integral, and $\alpha_{a_1 s_1,\dots, a_d s_d}^{(a_1+\cdots +a_d)s} \neq 0$, then
	\eq{\varphi(s) \geq \sum_{i=1}^{d} \frac{a_i}{a_1+\cdots +a_d}\varphi(s_i) .}{eq:ConvWRTBLineq2}
\end{prop}

\begin{proof}
    Note first that if $d=1$, the inequality trivially becomes an equality.
    For $d=2$, suppose we have $s_1$, $s_2$, $s$, $a_1$, and $a_2$ as in the proposition statement.
    Assume for now that $a_1$ and $a_2$ are integral.
    Then by \cite[Proposition~4.10, Theorem~6.1]{BLC} there exists some broken line segment $\gamma: [0, \tau] \to U^{\trop}(\Q)$ with $\gamma(0)=s_1$, $\gamma(\tau)=s_2$, and $\gamma\lrp{\frac{a_2}{a_1+a_2} \tau }= s$. 
    Next, if $a_1$ and $a_2$ are only rational, we can clear denominators, writing $a_1' = \lambda a_1$ and $a_2' = \lambda a_2$.
    By \thref{lem:rescale}, $\alpha_{a_1' s_1, a_2' s_2}^{(a_1'+a_2')s} \neq 0$ as well.
    Thus we obtain a broken line segment $\gamma: [0, \tau] \to U^{\trop}(\Q)$ with $\gamma(0)=s_1$, $\gamma(\tau)=s_2$, and $\gamma\lrp{\frac{a_2'}{a_1'+a_2'} \tau }= s$.
    Note however that $\frac{a_2'}{a_1'+a_2'} = \frac{a_2}{a_1+a_2}$, so we have precisely the same outcome as the case of integral coefficients.
    
    Since $\varphi$ is convex with respect to broken lines, we have
    \eqn{\varphi\lrp{\gamma\lrp{\frac{a_2}{a_1+a_2} \tau }} \geq \lrp{1 - \frac{a_2}{a_1+a_2}}\varphi(\gamma(0)) + \frac{a_2}{a_1+a_2} \varphi(\gamma(\tau)),}
    so
    \eqn{\varphi\lrp{s} \geq \frac{a_1}{a_1+a_2}\varphi(s_1) + \frac{a_2}{a_1+a_2} \varphi(s_2).}
    This establishes the claim for $d=2$.
    Next, suppose the claim holds for $d=k$.
    If $s_1, \dots, s_{k+1}, s$, $a_1,\dots, a_{k+1}$ are as in the proposition statement, then
    $\alpha_{a_1\, s_1,\dots,a_{k+1} \, s_{k+1}}^{\lrp{a_1+\cdots + a_{k+1}}s} \neq 0$.
    That is, 
    \eqn{\tf_{\lrp{a_1+\cdots + a_{k+1}}s} \text{ is a non-zero summand of and}\prod_{i=1}^{k+1}\tf_{a_i\, s_i}.}
    Expanding the first $k$ terms of the product and using linear independence of theta functions, we see that 
    \eqn{\tf_{\lrp{a_1+\cdots + a_{k+1}}s} \text{ must be a non-zero summand of }\tf_{\lrp{a_1+\cdots + a_{k}}s'}\tf_{a_{k+1}\, s_{k+1}}}
    for some $\lrp{a_1+\cdots + a_{k}}s'$ with $\alpha_{a_1\, s_1,\dots,a_{k} \, s_{k}}^{\lrp{a_1+\cdots + a_{k}}s'}\neq 0$.
    So, by the induction hypothesis we have
    \eqn{\varphi(s)&\geq \frac{a_1+\cdots + a_{k}}{a_1+\cdots + a_{k+1}} \varphi(s') + \frac{a_{k+1}}{a_1+\cdots + a_{k+1}} \varphi(s_{k+1})\\
    &\geq \frac{a_1+\cdots + a_{k}}{a_1+\cdots + a_{k+1}} \lrp{\sum_{i=1}^{k} \frac{a_i}{a_1+\cdots +a_k}\varphi(s_i)} + \frac{a_{k+1}}{a_1+\cdots + a_{k+1}} \varphi(s_{k+1})\\
    &= \sum_{i=1}^{k+1} \frac{a_i}{a_1+\cdots +a_{k+1}}\varphi(s_i)}
proving the claim.
\end{proof}

\begin{remark}\thlabel{rem:ConvWRTBL-Equiv}
    In fact, \thref{prop:ConvWRTBL} provides an equivalent characterization of functions $\varphi:S\to \Q$ which are convex with respect to broken lines.
    That is, we also have the opposite implication.
    Suppose for any $s_1,\dots, s_d, s \in S$, $a_1, \dots, a_d \in \Q_{\geq 0}$ with $a_1 s_1,\dots, a_d s_d$, and $(a_1+\cdots +a_d)s$ all integral, and $\alpha_{a_1 s_1,\dots, a_d s_d}^{(a_1+\cdots +a_d)s} \neq 0$, we have
    \eqn{\varphi(s) \geq \sum_{i=1}^{d} \frac{a_i}{a_1+\cdots +a_d}\varphi(s_i) .}
    Then we claim $\varphi$ is convex with respect to broken lines.
    To see this, consider a broken line segment $\gamma:[t_1, t_2]\to S$, and let $\overline{\gamma}$ be
    the reparametrized broken line segment $\overline{\gamma}:[0, \tau= t_2-t_1] \to S$ defined by $\overline{\gamma}(t)=\gamma(t-t_1)$.
    Clearly, 
    \eqn{\varphi(\gamma(t)) \geq \lrp{\dfrac{t_2-t}{t_2-t_1}} \varphi(\gamma(t_1)) + \lrp{\dfrac{t-t_1}{t_2-t_1}} \varphi(\gamma(t_2))} 
    for all $t\in [t_1,t_2]$
    if and only if
    \eqn{\varphi(\overline{\gamma}(t)) \geq \dfrac{\tau-t}{\tau} \varphi(\overline{\gamma}(0)) + \dfrac{t}{\tau} \varphi(\overline{\gamma}(\tau))}
    for all $t\in [0,\tau]$.
    We may always write $t= \frac{b}{a+b} \tau$.
    By \cite[Proposition~5.4, Theorem~6.1]{BLC}, we may choose $a$ and $b$ such that $a\, \overline{\gamma}(0)$, $b\, \overline{\gamma}(\tau)$, and $(a+b)\, \overline{\gamma}(\tau)$ are all integral and $\alpha_{a\, \overline{\gamma}(0),b\, \overline{\gamma}(\tau)}^{(a+b)\, \overline{\gamma}(\tau)} \neq 0$.
    Then \eqn{\varphi\lrp{\overline{\gamma}\lrp{\frac{b}{a+b} \tau}}
    &\geq
    \frac{a}{a+b}\varphi\lrp{\overline{\gamma}\lrp{0}}+ \frac{b}{a+b}\varphi\lrp{\overline{\gamma}\lrp{\tau}}\\
    &=\frac{\tau-t}{\tau}\varphi\lrp{\overline{\gamma}\lrp{0}}+ \frac{t}{\tau}\varphi\lrp{\overline{\gamma}\lrp{\tau}}
    .}
\end{remark}

\subsection{Equivalence of \texorpdfstring{\cite{GHKK}}{GHKK}'s ``min-convex'' and ``decreasing''}
\thref{prop:ConvWRTBL} and \thref{rem:ConvWRTBL-Equiv} allow us to resolve a question posed by Gross-Hacking-Keel-Kontsevich in \cite[Remark~8.5]{GHKK}, and we take a slight detour to do so here.

\begin{prop}\thlabel{prop:decreasing}
    A piecewise linear function $\varphi:U^{\trop}(\Q) \to \Q$ is convex with respect to broken lines if and only if it is {\it{decreasing}} in the sense of \cite[Definition~8.3]{GHKK}.
\end{prop}

\begin{proof}
	First, suppose $\varphi$ is convex with respect to broken lines.
	Let $s_1$, $s_2$, and $r$ be in $U^{\trop}(\Z)$ and satisfy $\alpha_{s_1, s_2}^r \neq 0$.
	We need to show that $\varphi(r)\geq \varphi(s_1) + \varphi(s_2)$.
	Comparing to \thref{prop:ConvWRTBL}, we have $a_1=a_2=1$, and $r=2 s$.
    Then 
    \eqn{\varphi(s) \geq \frac{1}{2} \varphi(s_1) +\frac{1}{2} \varphi(s_2) }
    and
    \eqn{\varphi(r) = \varphi(2 s) = 2 \varphi(s) \geq \varphi(s_1) + \varphi(s_2) .}
	That is, $\varphi$ is decreasing.

	For the other direction, we use an induction argument very similar to the one used in \thref{prop:ConvWRTBL}.
	Suppose $\varphi$ is decreasing.
	Let $s_1,\dots, s_d,s \in U^{\trop}(\Q)$ and $a_1,\dots, a_d \in \Q_{\geq 0}$, with $a_1 \, s_1, \dots, a_d \, s_d$, and $(a_1+\cdots + a_d)s$ all integral and
	$\alpha_{a_1 s_1, \dots, a_d s_d}^{(a_1+\cdots +a_d)s} \neq 0$.
	We need to show that \eqref{eq:ConvWRTBLineq2} holds for $\varphi$.
	For the $d=1$ case, \eqref{eq:ConvWRTBLineq2} trivially reduces to an equality.
	For $d=2$, since $\varphi$ is decreasing we have 
	\eqn{\varphi((a_1+a_2)s) \geq \varphi(a_1 s_1) + \varphi(a_2 s_2),}
	which implies 
	\eqn{\varphi(s) \geq \frac{a_1}{a_1+a_2}\varphi(s_1) + \frac{a_2}{a_1+a_2}\varphi(s_2)}
	by piecewise linearity of $\varphi$.
	So  \eqref{eq:ConvWRTBLineq2} holds for $d=2$.
	Now assume it holds for $d=k$, and consider the case $d=k+1$.
	As we argued in \thref{prop:ConvWRTBL}, 
    \eqn{\tf_{\lrp{a_1+\cdots + a_{k+1}}s} \text{ must be a non-zero summand of }\tf_{\lrp{a_1+\cdots + a_{k}}s'}\tf_{a_{k+1}\, s_{k+1}}}
    for some $\lrp{a_1+\cdots + a_{k}}s'$ with $\alpha_{a_1\, s_1,\dots,a_{k} \, s_{k}}^{\lrp{a_1+\cdots + a_{k}}s'}\neq 0$.
    So, by the induction hypothesis we have
    \eqn{\varphi(s)&\geq \frac{a_1+\cdots + a_{k}}{a_1+\cdots + a_{k+1}} \varphi(s') + \frac{a_{k+1}}{a_1+\cdots + a_{k+1}} \varphi(s_{k+1})\\
    &\geq \frac{a_1+\cdots + a_{k}}{a_1+\cdots + a_{k+1}} \lrp{\sum_{i=1}^{k} \frac{a_i}{a_1+\cdots +a_k}\varphi(s_i)} + \frac{a_{k+1}}{a_1+\cdots + a_{k+1}} \varphi(s_{k+1})\\
    &= \sum_{i=1}^{k+1} \frac{a_i}{a_1+\cdots +a_{k+1}}\varphi(s_i),}
	which proves the claim.
\end{proof}

\begin{prop}\thlabel{prop:min-convex}
	A piecewise linear function $\varphi:U^{\trop}(\Q) \to \Q$ is convex with respect to broken lines if and only if it is {\it{min-convex}} in the sense of \cite[Definition~8.2]{GHKK}.
\end{prop}

\begin{proof}
	First, suppose $\varphi$ is convex with respect to broken lines.
	We need to verify that $d\varphi$ is decreasing on $\dot{\gamma}$ for all broken lines $\gamma$.
	Suppose $\gamma$ crosses a wall at time $\tau$.
	Then for sufficiently small $\epsilon>0$, we have $\varphi(\gamma(\tau \pm \epsilon)) = \varphi(\gamma(\tau)) \pm \epsilon d\varphi_{\gamma(\tau\pm \epsilon)}(\dot{\gamma}(\tau\pm \epsilon))$ and 
	\eqn{\varphi(\tau) \geq \frac{1}{2}\lrp{\varphi(\tau)-\epsilon d\varphi_{\gamma(\tau-\epsilon)}(\dot{\gamma})} + \frac{1}{2} \lrp{\varphi(\tau)+\epsilon d\varphi_{\gamma(\tau+\epsilon)}(\dot{\gamma})}. }
	Simplifying, we find $d\varphi_{\gamma(\tau-\epsilon)}(\dot{\gamma}) \geq d\varphi_{\gamma(\tau+\epsilon)}(\dot{\gamma})$ as desired.

	The other direction follows from \cite[Lemma~8.4]{GHKK} and \thref{prop:decreasing}.
\end{proof}

Taken together, Propositions~\ref{prop:decreasing} and \ref{prop:min-convex} resolve a question posed in \cite[Remark~8.5]{GHKK}:

\begin{cor}\thlabel{cor:min-convex=decreasing}
    The notions ``min-convex'' and ``decreasing'' of \cite[Definitions~8.2~\&~8.3]{GHKK} are equivalent.
\end{cor}

\subsection{Basic results}
We now state and prove some basic results about functions which are convex with respect to broken lines.

\begin{lemma}\thlabel{lem:SumConvWRTBL}
    Let $\varphi_1,\, \varphi_2:S \to \Q$ be convex with respect to broken lines.  Then $\varphi_1+ \varphi_2$ is convex with respect to broken lines.
\end{lemma}

\begin{proof}
    This follows immediately from \thref{def:ConvWRTBL}.
\end{proof}

\begin{prop}\thlabel{prop:ConvWRTBL-BLC}
    Let $\varphi:U^{\trop}(\Q) \to \Q$ be convex with respect to broken lines.
    Then \eqn{\Xi_{\varphi,r}:=\lrc{x\in U^{\trop}(\Q)\, :\, \varphi(x) \geq -r}} is broken line convex.
\end{prop}

\begin{proof}
By \thref{char_blc}, this holds if and only if $\Xi_{\varphi,r} = t\, \Xi_{\varphi,r} \+t (1-t) \Xi_{\varphi,r}$.
We always have the inclusion  $\Xi_{\varphi,r} \subset t\, \Xi_{\varphi,r} \+t (1-t) \Xi_{\varphi,r}$, so we just need to show the opposite inclusion.
Let $z\in t\, \Xi_{\varphi,r} \+t (1-t) \Xi_{\varphi,r}$.
Then there exists $x\in t\, \Xi_{\varphi,r}$, $y\in (1-t) \Xi_{\varphi,r}$, and $a \in \Z_{>0}$ such that $a x$, $a y$, and $a z$ are all integral and $\alpha_{a x, a y}^{a z}\neq 0$.
Define $x'$, $y' \in \Xi_{\varphi,r}$ by $x=t x'$, $y= (1-t)y'$.
Now, $a= t a + (1-t) a$, so $0\neq \alpha_{a x, a y}^{a z} =\alpha_{a t x', a (1-t) y'}^{a z}$.
Then by \thref{prop:ConvWRTBL},
\eqn{\varphi(z) &\geq t \varphi(x') + (1-t) \varphi(y')\\
&\geq t(-r) + (1-t) (-r)\\
&= -r.}
That is, $z\in \Xi_{\varphi,r}$.
\end{proof}

\begin{lemma}\thlabel{lem:linearwrtblc}
    Let $\varphi:U^{\trop}(\Q) \to \Q$ be convex with respect to broken lines, and let $\gamma:[t_1,t_2] \to U^{\trop}(\Q)$ be a broken line segment satisfying 
    \eqn{\varphi(\gamma(t)) = \lrp{\frac{t_2 -t}{t_2 -t_1}} \varphi(\gamma(t_1)) + \lrp{\frac{t-t_1}{t_2 -t_1}} \varphi(\gamma(t_2))}
    for some $t \in \lrp{t_1,t_2}$.
    Then 
    \eqn{\varphi(\gamma(t)) = \lrp{\frac{t_2 -t}{t_2 -t_1}} \varphi(\gamma(t_1)) + \lrp{\frac{t-t_1}{t_2 -t_1}} \varphi(\gamma(t_2))}
    for all $t \in \lrb{t_1,t_2}$.
\end{lemma}

\begin{proof}
    Suppose not.  Then there is some $t'\in \lrp{t_1,t_2}$\footnote{We take the open interval here since equality is clear for the endpoints $t_1$ and $t_2$.}, $t'\neq t$,  with 
    \eq{\varphi(\gamma(t')) > \lrp{\frac{t_2 -t'}{t_2 -t_1}} \varphi(\gamma(t_1)) + \lrp{\frac{t'-t_1}{t_2 -t_1}} \varphi(\gamma(t_2)).}{eq:StrictIneq}
    The argument is identical for $t'<t$ and $t'>t$, so without loss of generality, take $t'<t$.
    Since $\varphi$ is convex with respect to broken lines, by restricting $\gamma$ to $\lrb{t',t_2}$ we find
    \eqn{\varphi(\gamma(t)) \geq \lrp{\frac{t_2 -t}{t_2 -t'}} \varphi(\gamma(t')) + \lrp{\frac{t-t'}{t_2 -t'}} \varphi(\gamma(t_2)).}
    That is,
    \eqn{\lrp{\frac{t_2 -t}{t_2 -t_1}} \varphi(\gamma(t_1)) + \lrp{\frac{t-t_1}{t_2 -t_1}} \varphi(\gamma(t_2)) \geq \lrp{\frac{t_2 -t}{t_2 -t'}} \varphi(\gamma(t')) + \lrp{\frac{t-t'}{t_2 -t'}} \varphi(\gamma(t_2)),}
    which upon simplifying yields
    \eqn{\lrp{\frac{t_2 -t'}{t_2 -t_1}} \varphi(\gamma(t_1)) + \lrp{\frac{t'-t_1}{t_2 -t_1}} \varphi(\gamma(t_2)) \geq  \varphi(\gamma(t')).}
    This contradicts the strict inequality \eqref{eq:StrictIneq}.
\end{proof}

One type of function that will come up frequently in the remainder of the paper is simply given by evaluation: $\lra{\ \boldsymbol{\cdot}\ ,y}: U^{\trop}(\Q) \to \Q$.
For this reason, we introduce the following terminology.

\begin{definition}
    We say a function $\varphi: U^{\trop}(\Q) \to \Q$ is {\it{tropically linear}} if $\varphi = \lra{\ \boldsymbol{\cdot}\ ,y}$ for some $y\in (U^\vee)^{\trop}(\Q)$.
    We also use the terminology for a function $\psi$ on a $\Q_{\geq 0}$-invariant subset $\sigma$ in $U^{\trop}(\Q)$ if there exists an extension of $\psi$ from $\sigma$ to $U^{\trop}(\Q)$ which is tropically linear.
\end{definition}

The following results are a corollaries of \thref{thm:ValuativeIndependence}.

\begin{cor}\thlabel{cor:TropLinearEquality}
    Let $\varphi: U^{\trop}(\Q) \to \Q$ be tropically linear, and 
    consider a collection of integral tropical points $x_1, \dots, x_d \in U^{\trop}(\Z)$. 
    Then
    \eqn{\sum_{i=1}^d\varphi(x_i) = \min \lrc{\varphi(x) : x\in U^{\trop}(\Z), \alpha_{x_1, \dots, x_d}^{x}\neq 0}.}
\end{cor}

\begin{proof}
    Let $\varphi = \lra{\, \boldsymbol{\cdot} \ ,y_\varphi}$,
    and let $a y_\varphi$ be integral for some $a>0$.
    Then 
    \eqn{ a y_\varphi \lrp{\tf_{x_1}} + \cdots + a y_\varphi \lrp{\tf_{x_d}} = a y_\varphi \lrp{\tf_{x_1} \cdots \tf_{x_d}} = \min\lrc{a y_\varphi (\tf_x) : x\in U^{\trop}(\Z), \alpha_{x_1, \dots, x_d}^{x}\neq 0}. }
    But then 
    \eqn{\sum_{i=1}^d\varphi(x_i) = \min \lrc{\varphi(x) : x\in U^{\trop}(\Z), \alpha_{x_1, \dots, x_d}^{x}\neq 0}}
    as claimed.
\end{proof}

\begin{cor}\thlabel{cor:TropLinCWRTBL}
    A tropically linear function is convex with respect to broken lines.
\end{cor}

\begin{proof}
    Let $\varphi: U^{\trop}(\Q) \to \Q$ be tropically linear, and consider any $x_1,\dots, x_d, x \in U^{\trop}(\Q)$, $a_1, \dots, a_d \in \Q_{\geq 0}$ with $a_1 x_1,\dots, a_d x_d$, and $(a_1+\cdots +a_d)x$ all integral, and $\alpha_{a_1 x_1,\dots, a_d x_d}^{(a_1+\cdots +a_d)x} \neq 0$.
    By \thref{cor:TropLinearEquality}, 
    \eqn{\sum_{i=1}^d\varphi(a_i x_i) = \min \lrc{\varphi(s) : s\in U^{\trop}(\Z), \alpha_{a_1 x_1,\dots, a_d x_d}^{s}\neq 0}.}
    So, we have that 
    \eqn{ \varphi\lrp{(a_1+\cdots +a_d)x} \geq \sum_{i=1}^d\varphi(a_i x_i),}
    and
    \eqn{ \varphi\lrp{x} \geq \sum_{i=1}^d\frac{a_i}{a_1+\cdots +a_d} \varphi(x_i).}
    By \thref{rem:ConvWRTBL-Equiv}, $\varphi$ is convex with respect to broken lines.
\end{proof}

\section{Broken line convex polyhedral geometry}\label{sec:Polyhedra}

Let $U$ be a cluster variety for which the full Fock-Goncharov conjecture holds and let $U^\vee$ be its Fock-Goncharov dual.

We borrow some notation from \cite{Brondsted}.
\begin{definition}\thlabel{def:Half-space-Hyperplane}
For $y\in (U^\vee)^\trop(\Q)$ and $r\in \Q$ denote by $K(y,r)$ the set $\lrc{x \in U^\trop(\Q) : \lra{x, y } \geq -r }$.
We call $K(y,r)$ a {\it{tropical half-space}}, and we call its boundary $H(y,r):= \lrc{x \in U^\trop(\Q) : \lra{x, y } = -r }$ a {\it{tropical hyperplane}}.
For $S\subset U^\trop(\Q)$, we say  $K(y,r)$ is a {\it{supporting tropical half-space for $S$}} and $H(y,r)$ is a {\it{supporting tropical hyperplane for $S$}} if $S \subset K(y,r)$ and $S\cap H(y,r) \neq \varnothing$.
We define tropical half-spaces and hyperplanes in $(U^\vee)^\trop(\Q)$ analogously.
\end{definition}

\begin{remark}\thlabel{rem:Rescale}
    As $\lra{\ \boldsymbol{\cdot}\ , a y} = a \lra{\ \boldsymbol{\cdot}\ , y}$ for all $a > 0$, we have that $K(y,r) = K(ay,ar)$ for all $a>0$.
\end{remark}

\begin{lemma}\thlabel{lem:TropHalfSpaceBLC}
    A tropical half-space is broken line convex.
\end{lemma}

\begin{proof}
    By \thref{cor:TropLinCWRTBL}, a tropically linear function is convex with respect to broken lines.
    Then the claim follows from \thref{prop:ConvWRTBL-BLC}.
\end{proof}

\begin{definition}\thlabel{def:polyhedral}
    A subset $S\subset U^{\trop}(\Q)$ is {\it{polyhedral}} if 
    \eqn{S = \bigcap_{i \in I} K(y_i,r_i)}
    for some finite indexing set $I$.
    We will always take $y_i \in (U^{\vee})^{\trop}(\Q)$ and $r_i \in \Q$.\footnote{In usual convex geometry, this reduces to the notion of ``rational polyhedral''.  As we only work in the rational setting in this paper, we drop the ``rational'' descriptor from our terminology here.} 
    If additionally $S$ is bounded, we say it is {\it{polytopal}}.
\end{definition}

\subsection{Faces}

\begin{definition}\thlabel{def:face-half-space}
Let $S\subset U^{\trop}(\Q)$ be broken line convex.
We say that a subset $F$ of $S$ is a {\it{face}} of $S$ if there is a  
tropical half-space $K(y,r)\supset S$ with $F= S \cap H(y,r)$.
We say this face $F$ is a {\it{proper}} face if $F\notin \lrc{\varnothing, S}$.
We call $0$-dimensional faces {\it{vertices}}, $1$-dimensional faces {\it{edges}}, and codimension 1 faces {\it{facets}}.
By convention, we view $\varnothing$ as a $-1$-dimensional face.
We denote the set of faces of $S$ by $\mathcal{F}_S$.
\end{definition}

\begin{remark}
    We will typically discuss faces of polyhedral sets rather than arbitrary broken line convex sets.
    However, the definition makes sense for arbitrary broken line convex sets, and we will want to use the face terminology for certain sets prior to proving that they are in fact polyhedral.
\end{remark}

\begin{center}
\begin{minipage}{.85\linewidth}
\begin{center}
{\bf{Warning:}} {\emph{Unlike in usual convex geometry, faces in broken line convex geometry need not be broken line convex.}}
\end{center}
\end{minipage}
\end{center}

For an example of this phenomenon, see Figure~\ref{fig:Non-convexFace}.

\noindent
\begin{center}
\begin{minipage}{.85\linewidth}
\captionsetup{type=figure}
\begin{center}
\begin{tikzpicture}[scale=.85]

    \def\x{1.5}
    \def\d{1}
    \def\l{3}
    \def\op{.3}
    
    \path (-\l,0) coordinate (3) --++ (\l,0) coordinate (0) --++ (\l,0) coordinate (1);
    \path (0,\l) coordinate (2) --++ (0,-2*\l) coordinate (4) --++ (\l,0) coordinate (5);
    \path (5) --++ (0,2*\l) coordinate (tr) --++ (-2*\l,0) coordinate (tl) --++ (0,-2*\l) coordinate (bl);

    \draw[thick, ->] (3) -- (1);
    \draw[thick, ->] (2) -- (4);
    \draw[thick, ->] (0) -- (5);

    \node at (.9,\l) {$1+z^{(0,1)}$};
    \node at (-2.3,-.35) {$1+z^{(-1,0)}$};
    \node at (3.25,-2) {$1+z^{(-1,1)}$};

    \coordinate (v1) at (-\x,\x) {};
    \coordinate (v2) at (0,\x) {};
    \coordinate (v3) at (\x,0) {};
    \coordinate (v4) at (\x,-\x) {};
    \coordinate (v5) at (0,-\x) {};
    \coordinate (v6) at (-\x,0) {};

    \path [blue-ish, thick, fill= blue-ish, fill opacity=\op] (v1.center)--(v2.center)--(v3.center)--(v4.center)--(v5.center)--(v6.center)--cycle;

    \draw[color= blue-ish, thick] (v3.center)--(v4.center)--(v5.center)--(v6.center)--(v1.center);

    \draw[color= orange, very thick] (v1.center)--(v2.center)--(v3.center);

    \path (v2) --++ (.4*\x,0) node [color=orange] {$F$};

\end{tikzpicture}

\captionof{figure}{\label{fig:Non-convexFace}A polytopal set $\tc{blue-ish}{S}\subset (\cA^\vee)^{\trop}(\Q)$ for the $\cA$ cluster variety of type $A_2$. The indicated face \tc{orange}{$F$} is not broken line convex. 
As is standard, to draw this picture we identify $(\cA^\vee)^{\trop}(\Q)$ with $\Q^2$ via a choice of seed.
} 
\end{center}
\end{minipage}
\end{center}

However, faces do satisfy some weaker notion of convexity.
To motivate this weaker convexity notion, we make an observation about tropical hyperplanes.

\begin{prop}\thlabel{cor:HyperplaneBrokenLines}
    Every pair of points $x_1$, $x_2$ in a tropical hyperplane $H(y,r)$ is connected by a broken line segment $\gamma$ whose support is contained in $H(y,r)$.
\end{prop}

\begin{proof}
    This is a simple corollary of \thref{thm:ValuativeIndependence} and \thref{lem:linearwrtblc}.
    By the \thref{thm:ValuativeIndependence}, there exists a broken line segment $\gamma:[t_1,t_2] \to U^{\trop}(\Q)$ with endpoints $x_1$ and $x_2$
    such that 
    \eqn{\lra{\gamma\lrp{\frac{t_1+t_2}{2}},y} &= \frac{1}{2}\lra{\gamma\lrp{t_1},y}+\frac{1}{2}\lra{\gamma\lrp{t_2},y}\\
    &= \frac{1}{2}\lra{x_1,y}+\frac{1}{2}\lra{x_2,y}\\
    &= -r.}
    Then by \thref{lem:linearwrtblc}, 
    \eqn{\varphi(\gamma(t)) &= -\lrp{\frac{t_2 -t}{t_2 -t_1}} r - \lrp{\frac{t-t_1}{t_2 -t_1}} r\\
    &=-r}
    for all $t \in \lrb{t_1,t_2}$.
    That is, the support of $\gamma$ is contained in $H(y,r)$.
\end{proof}

In light of \thref{cor:HyperplaneBrokenLines}, we make the following definition.

\begin{definition}\thlabel{def:WeaklyConvex}
    We say a subset $S\subset U^{\trop}(\Q)$ is {\it{weakly convex}} if for every pair of points $s_1,\, s_2 \in S$, there exists a broken line segment with endpoints $s_1$ and $s_2$ whose support is contained in $S$.    
\end{definition}

Clearly, the notions {\it{broken line convexity}} and {\it{weak convexity}} coincide in usual convex geometry.
They are very different notions in $U^{\trop}(\Q)$, but both play important roles in in the theory of broken line convex geometry.
In fact, the two convexity notions interact with each other:

\begin{prop}\thlabel{prop:BLCcapWC}
    Let $S$ and $S'$ be subsets of $U^{\trop}(\Q)$ with $S$ broken line convex and $S'$ weakly convex.
    Then $S \cap S'$ is weakly convex.
\end{prop}

\begin{proof}
    Let $s_1,\, s_2 \in S\cap S'$.
    Since $S'$ is weakly convex, there exists a broken line segment $\gamma$ with endpoints $s_1$ and $s_2$ whose support is contained in $S'$.
    Since $S$ is broken line convex, the support of $\gamma$ must be contained in $S$ as well.
    Hence the support of $\gamma$ is contained in $S\cap S'$, proving the claim.
\end{proof}

\begin{cor}\thlabel{cor:FaceWC}
    Every face of a polyhedral set is weakly convex.
\end{cor}

\begin{proof}
    By definition, a face $F$ of a polyhedral set $S\subset U^{\trop}(\Q)$ is of the form 
    \eqn{F = S \cap H(y,r)}
    for $H(y,r)$ a tropical hyperplane at the boundary of a tropical half-space $K(y,r)$ which contains $S$.
    The polyhedral set $S$ is broken line convex, and by \thref{cor:HyperplaneBrokenLines}, $H(y,r)$ is weakly convex.
\end{proof}

In usual convex geometry, the set of faces of a polyhedron forms a polyhedral complex.
Unfortunately, in general the faces of a polyhedral set in $U^{\trop}(\Q)$ will not form such a complex.
For instance, if we consider the bigon of Figure~\ref{fig:Bigon1}, the intersection of the pair of facets is a pair of vertices-- so in this instance the intersection of two faces is {\emph{not}} a face, but rather a union of faces.

\noindent
\begin{center}
\begin{minipage}{.85\linewidth}
\captionsetup{type=figure}
\begin{center}
\begin{tikzpicture}[scale=.85]

    \def\x{1.5}
    \def\l{3}
    \def\op{.3}

    \path (-\l,0) coordinate (3) --++ (\l,0) coordinate (0) --++ (\l,0) coordinate (1);
    \path (0,\l) coordinate (2) --++ (0,-2*\l) coordinate (4) --++ (\l,0) coordinate (5);
    \path (5) --++ (0,2*\l) coordinate (tr) --++ (-2*\l,0) coordinate (tl) --++ (0,-2*\l) coordinate (bl)--++ (2*\l,0) coordinate (br);

    \draw[thick, ->] (3) -- (1);
    \draw[thick, ->] (2) -- (4);
    \draw[thick, ->] (0) -- (5);

    \node at (.9,\l) {$1+z^{(0,1)}$};
    \node at (-2.3,-.35) {$1+z^{(-1,0)}$};
    \node at (3.25,-2) {$1+z^{(-1,1)}$};

    \coordinate (v1) at (-\x,\x);

    \path (v1) --++ (-.2*\x,.25*\x) node[color = more-blue] {$(-1,1)$};
    
    \coordinate (v2) at (\x,-\x);

    \path (v2) --++ (.5*\x,.2*\x) node[color = more-green] {$(1,-1)$};

    \coordinate (top-bend) at (0,0.5*\x);
    \coordinate (bottom-bend) at (-0.5*\x,0);
    
    \path [fill= blue-ish, fill opacity=\op] (v1.center)--(top-bend)--(v2.center)--(bottom-bend)--cycle;

    \draw[orange, thick] (v1.center)--(top-bend)--(v2.center);

    \draw[purple, thick] (v1.center)--(bottom-bend)--(v2.center);

    \node [circle, fill, inner sep = 1.5pt, color = more-blue] at (v1) {};

    \node [circle, fill, inner sep = 1.5pt, color = more-green] at (v2) {};
    
\end{tikzpicture}

\captionof{figure}{\label{fig:Bigon1}A bigon \tc{blue-ish}{$S$} in $(\cA^\vee)^{\trop}(\Q)$ together with its faces $\mathcal{F}_{\tc{blue-ish}{S}}$ for the $\cA$ cluster variety of type $A_2$. Note that the intersection of the facets is a pair of vertices rather than a single face.}
 
\end{center}
\end{minipage}
\end{center}

Nevertheless, the set of faces of a polyhedral set has a structure very reminiscent of a polyhedral complex.
To make this precise, we introduce the following definition:

\begin{definition}
    \thlabel{def:pseudo-complex}Let $\mathcal{P}$ be a set of subsets of $U^{\trop}(\Q)$.
    We say that $\mathcal{P}$ is a pseudo-complex if it has the following properties:
    \begin{enumerate}
        \item \label{it:PBoundary}If $P \in \mathcal{P}$, then there is a subset $\mathcal{A}$ of $\mathcal{P}$ with
        \eqn{\partial P = \bigcup_{P' \in \mathcal{A}} P'.}
        \item \label{it:PIntersection}If $P_1,\, P_2 \in \mathcal{P}$, then there is a subset $\mathcal{B}$ of $\mathcal{P}$ with 
        \eqn{P_1 \cap P_2 = \bigcup_{P \in \mathcal{B}} P.}
        \item \label{it:PIntersectionInBoundary}If $P_1,\, P_2 \in \mathcal{P}$ and $P_1 \cap P_2 \subsetneq P_1$, then 
        \eqn{P_1 \cap P_2 \subset \partial P_1.} 
    \end{enumerate}
\end{definition}

\begin{prop}\thlabel{prop:FacePseudo-Complex}
    The set of faces $\mathcal{F}_S$ of a polyhedral set $S\subset U^{\trop}(\Q)$ forms a pseudo-complex.
\end{prop}

To establish \thref{prop:FacePseudo-Complex},
we will need a pair of lemmas:

\begin{lemma}\thlabel{lem:HSpaceOfMinkSum}
    Let $S = \bigcap_{i \in I} K(y_i,r_i)$, and let $\displaystyle{y\in \Sumt_{i\in I} y_i}$ and $\displaystyle{r= \sum_{i \in I}r_i}$.
    Then $S \subset K(y,r)$ and
    \eqn{S \cap H(y,r)  \subset \lrp{S \cap \bigcap_{i \in I} H(y_i,r_i)}. }
\end{lemma}

\begin{proof}
    Let $x\in S$, then we have that $\lra{x,y_i} \geq -r_i$.
    By \thref{prop:ConvWRTBL}, 
    \eqn{\lra{x,y}\geq  \sum_{i\in I}\lra{x,y_i} \geq - \sum_{i\in I}r_i = -r.}
    So $x\in K(y,r)$.

    Now let $x\in S \cap H(y,r)$.  Then 
    \eqn{-r = \lra{x,y}\geq  \sum_{i\in I}\lra{x,y_i} \geq - \sum_{i\in I}r_i = -r,}
    and we must have equality throughout.
    That is, $ \lra{x,y_i} = -r_i$ and $x \in H(y_i,r_i)$ for all $i\in I$.
\end{proof}

\begin{lemma}\thlabel{prop:BLInHyperplanes}
    Let $\gamma:[t_1,t_2] \to U^{\trop}(\Q)$ be a broken line segment whose support is contained in $H(y_1,r_1) \cap K(y_2,r_2)$.
    If $\supp(\gamma) \cap H(y_2,r_2)$ is one dimensional, then $\supp(\gamma) \subset H(y_2,r_2)$.
\end{lemma}

\begin{proof}
    By restricting the domain of $\gamma$ and reversing the direction of $\gamma$ as needed, we may reduce to the case in which
    $\gamma([t_1,\tau]) \subset H(y_1,r_1) \cap H(y_2,r_2)$ for some $\tau \in (t_1, t_2)$.
    Without hitting a wall, there is no way for $\gamma$ to leave the intersection $H(y_1,r_1) \cap H(y_2,r_2)$, so suppose $\gamma$ crosses the wall $(\wall, f_\wall(z^{m_\wall}))$, with $\wall \subset n_{\wall}^{\perp}$, at time $\tau$.
    There are three possibilities for $\gamma(\tau+\epsilon)$ for small $\epsilon>0$.
    \begin{enumerate}
        \item \label{it:=r}$\gamma(\tau+\epsilon) \in H(y_1,r_1) \cap H(y_2,r_2)$
        \item \label{it:<r}$\lra{\gamma(\tau+\epsilon), y_2} < -r_2$        
        \item \label{it:>r}$\lra{\gamma(\tau+\epsilon), y_2} > -r_2$        
    \end{enumerate}
    We want eliminate Items~\ref{it:<r} and \ref{it:>r}. 
    We immediately note that if $\lra{\gamma(\tau+\epsilon), y_2} < -r_2$, then $\supp(\gamma) \not\subset K(y_2,r_2)$, eliminating Item~\ref{it:<r}.
    
    Next, suppose $\lra{\gamma(\tau+\epsilon), y_2} > -r_2$.
    Denote the velocity of $\gamma$ immediately prior to crossing $(\wall,f_{\wall}(z^{m_\wall}))$ by $\dot{\gamma}_-$ and the velocity immediately after crossing by $\dot{\gamma}_+$.
    For some $k\geq 0$, we have $\dot{\gamma}_+=\dot{\gamma}_- -k m_{\wall}$. 
    We can give a new broken line segment
    $\gamma':[t_1', t_2']\to U^{\trop}(\Q)$ crossing $(\wall,f_{\wall}(z^{m_\wall}))$ such that for some $\lambda>0$ and some small $\delta>0$ 
    \begin{itemize}
        \item $\gamma'(t_1')= \gamma(t_1)$, 
        \item $\dot{\gamma}'_-= \lambda \lrp{\dot{\gamma}_- - \delta m_{\wall}}$,  
        \item $\dot{\gamma}_+ = \lambda \lrp{\dot{\gamma}_+ - \delta m_{\wall}}$, and
        \item $\lra{\gamma'(t_2'), y_2} > -r_2$.
    \end{itemize}
    The factor of $\lambda$ above is simply to ensure we can make exponent vectors integral. With this in mind, since $\dot{\gamma}_+=\dot{\gamma}_- -k m_{\wall}$ pertains to an allowed bend and $\lra{n_\wall, m_{\wall}}=0$, for some $\lambda>0$ we have that 
    \eqn{\dot{\gamma}'_+= \lambda \lrp{\dot{\gamma}_+ - \delta m_{\wall}} =  \lambda \lrp{(\dot{\gamma}_- -km_{\wall})-\delta m_{\wall}} =   \lambda \lrp{(\dot{\gamma}_- -\delta m_{\wall})-k m_{\wall}} =  \dot{\gamma}'_- - \lambda k m_{\wall}}
    is also an allowed bend.
    See Figure~\ref{fig:ContradictConvexity} for an illustration of this scenario.
    
\noindent
\begin{center}
\begin{minipage}{.85\linewidth}
\captionsetup{type=figure}
\begin{center}
\begin{tikzpicture}[scale=.85]

    \def\x{1.5}
    \def\l{3}
    \def\d{.07}
    \def\k{.4}
    \def\v{2}
    \def\op{.3}

    \path (-\l,0) coordinate (3) --++ (\l,0) coordinate (0) --++ (\l,0) coordinate (1);
    \path (0,\l) coordinate (2) --++ (0,-2*\l) coordinate (4) --++ (\l,0) coordinate (5);
    \path (5) --++ (0,2*\l) coordinate (tr) --++ (-2*\l,0) coordinate (tl) --++ (0,-2*\l) coordinate (bl)--++ (2*\l,0) coordinate (br);

    \path [name path = right] (tr) -- (br); 

    \path [name path = bottom] (bl) -- (br); 

    \path [name path = top] (tl) -- (tr); 

    \path [name path = after-wall] (0) --++ (\l,-0.5*\l);

    \path [name path = wall-above] (0) --++ (.2*\l,\l);

    \path [name path = wall-below] (0) --++ (-.2*\l,-\l);

    \path [name intersections={of=right and after-wall, by=right-end}];

    \path [name intersections={of=top and wall-above, by=wall-top}];

    \path [name intersections={of=bottom and wall-below, by=wall-bot}];
    
    \path[fill, color = blue-ish, opacity = \op] (3) -- (0) -- (right-end) -- (br) -- (bl) -- cycle; 
    
    \draw[blue-ish, thick] (3) -- (0) -- (right-end) node [pos=.7, sloped, above, color= blue-ish] {$K(y_2,r_2)$};

    \draw[thick, dashed] (wall-bot) -- (wall-top);

    \path (wall-top) --++ (.8*\x,-.4*\x) node {$\lrp{\wall,f_{\wall}(z^{m_\wall})}$};

    \node [circle, fill, inner sep = 1.5pt, color = orange](v1) at (-.7*\l,0) {};

    \node [circle, fill, inner sep = 1.5pt, color = orange](v2) at (\v-\k,-5*\k) {};   

    \draw[very thick, orange] (v1) -- (0) -- (v2) node [pos=.6, sloped, below, color= orange] {$\left.\gamma\right|_{[\tau-\epsilon,\tau+\epsilon]}$};

    \path [name path = gpm] (v1) --++(2*\v+2*\d,10*\d);

    \path [name intersections={of=gpm and wall-above, by=bend}];

    \draw[very thick, purple] (v1) --(bend)node [pos=.6, sloped, above, color= purple] {$\gamma'$}--++(\v-\k+\d,-5*\k+5*\d) node [circle, fill, inner sep = 1.5pt, color = purple] {};

\end{tikzpicture}

\captionof{figure}{\label{fig:ContradictConvexity}
Schematic of \tc{orange}{$\left.\gamma\right|_{[\tau-\epsilon,\tau+\epsilon]}$} and \tc{purple}{$\gamma'$} as detailed above.
}
 
\end{center}
\end{minipage}
\end{center}
    However, this broken line segment $\gamma'$ has endpoints in $K(y_2,r_2)$, while having support not contained in $K(y_2, r_2)$.
    In particular, the point at which $\gamma'$ crosses $(\wall,f_{\wall}(z^{m_\wall}))$ does not lie in $K(y_2, r_2)$.
    This contradicts broken line convexity of $K(y_2, r_2)$.
\end{proof}

\begin{proof}[Proof of \thref{prop:FacePseudo-Complex}]
    We begin with Item~\ref{it:PBoundary} of \thref{def:pseudo-complex}.
    For the face $F =\varnothing$, the statement is vacuous.
    So consider a face $F\neq \varnothing$. 
    Then $F$ is of the form $F=H(y,r)\cap S$ for some supporting tropical hyperplane $H(y,r)$.   Let 
    \eqn{S= \bigcap_{i \in I} K(y_i,r_i)}
    be a presentation of $S$.
    Necessarily, the boundary of $F$ is obtained by intersection with some of the tropical hyperplanes $H(y_i,r_i)$.
    Precisely, define 
    \eqn{\mathcal{I}_F :=  \lrc{J \subset I \, : \,  \lrp{F \cap \bigcap_{j \in J} H(y_j,r_j) }\subsetneq F}.}
    Then
    \eqn{\partial F =  F \cap \bigcup_{J \in \mathcal{I}_F }\bigcap_{j \in J} H(y_j,r_j). }
    For shorthand, write 
    \eqn{F_J:= \lrp{F \cap \bigcap_{j \in J} H(y_j,r_j) },}
    so $\partial F = \bigcup_{J \in \mathcal{I}_F} F_J$.
    Observe that 
    \eqn{F_J= \lrp{H(y,r) \cap \bigcap_{j \in J} H(y_j,r_j) }\cap S.}
    Now set $r'=r+\sum_{j\in J}r_j$ and let $y'\in y\+t \Sumt y_j$.
    Then by \thref{lem:HSpaceOfMinkSum}, 
    \eqn{\lrp{K(y,r) \cap \bigcap_{j \in J} K(y_j,r_j) } \subset K(y',r')}
    and
    \eqn{S \cap H(y',r') \subset F_J. }
    Moreover, by \thref{cor:TropLinearEquality}, for each $x\in F_J$, there exists some such $y'_x\in y\+t \Sumt y_j$ with $H(y'_x,r')$ a supporting tropical hyperplane containing $x$.
    That is, $F_x:= H(y'_x,r')\cap S$ is a face of $S$ containing $x$.
    This establishes Item~\ref{it:PBoundary}.

    Now we turn our attention to Item~\ref{it:PIntersection} of \thref{def:pseudo-complex}, whose proof is very similar to the one above.
    As before, if $F_1 \cap F_2 = \varnothing$, the claim trivially holds.
    Suppose $F_1 \cap F_2 \neq \varnothing$, and
    let $F_i = H(y_i,r_i)\cap S$, where $H(y_i,r_i)$ is a supporting tropical hyperplane.
    Now \thref{lem:HSpaceOfMinkSum} and \thref{cor:TropLinearEquality} imply that for each $x \in F_1 \cap F_2$, there is a supporting tropical hyperplane $H(y_x,r_1+r_2)$ containing $x$ such that $K(y_1,r_1)\cap K(y_2,r_2) \subset K(y_x,r_1+r_2)$ and $F_x:=S\cap H(y_x,r_1+r_2) \subset \lrp{S \cap H(y_1,r_1)\cap H(y_2,r_2)} = F_1 \cap F_2$, establishing Item~\ref{it:PIntersection}.

    Finally, for Item~\ref{it:PIntersectionInBoundary} of \thref{def:pseudo-complex},
    let $F_i = H(y_i,r_i) \cap S$.
    Since $F_1 \cap F_2$ is properly contained in $F_1$, for some $x\in F_1$ we have $\lra{x,y_2} > -r_2$.
    However, tropically linear functions are continuous so this implies $\lra{\ \boldsymbol{\cdot} \ ,y_2} > -r_2$ on an open neighborhood of $x$.
    Now consider a broken line segment contained in $F_1$ which begins at $x$ and proceeds to some $x'$ with $\lra{x',y_2}=-r_2$.
    (If no such broken line segment exists, then $F_1 \cap F_2 = \varnothing$, and we are done.)
    By \thref{prop:BLInHyperplanes}, this broken line segment {\emph{cannot}} be extended in such a way that a positive length subsegment lies in $H(y_1,r)\cap H(y_2,r)$.
    Then $x'$ must in fact lie at the boundary of $F_1$.
\end{proof}

\begin{prop}\thlabel{prop:FaceComplement}
    Let $F$ be a proper face of a polyhedral subset $S \subset U^{\trop}(\Q)$.
    Then $F$ is not contained in ${\bconv(S\setminus F)}$.
\end{prop}

\begin{proof}
    $F$ is of the form $F=H(y,r)\cap S$ for some supporting tropical hyperplane $H(y,r)$ for $S$.
    The open tropical half-space $K(y,r)\setminus H(y,r)$ is broken line convex, so its intersection with $S$ is as well.
    But $F$ is not contained in this intersection.
\end{proof}

\subsection{The weak face fan}

The other vitally important polyhedral complex in the theory of toric varieties is the fan.
To pursue our goal of a cluster version of Batyrev-Borisov duality in future work, we will primarily be interested in a particular sort of fan, namely  a face fan.
So, we now turn our attention to defining the broken line convex geometry analogue of a face fan, and showing that it also forms a pseudo-complex.

\begin{prop}\thlabel{prop:WeakCone(S)}
    If $S \subset U^{\trop}(\Q)$ is weakly convex, then so is $\Q_{\geq 0} \cdot S$.    
\end{prop}

\begin{proof}
    Consider an arbitrary pair of points $\lambda_1 s_1,\ \lambda_2 s_2 \in \Q_{\geq 0} \cdot S$.
    Let $\gamma:[0,T] \to U^{\trop}(\Q)$ be a broken line segment with endpoints $s_1$ and $s_2$ whose support is contained in $S$.
    We will show that there is a broken line segment $ \widetilde{\gamma}:[0, \widetilde{T}] \to U^{\trop}(\Q)  $ with endpoints $\lambda_1 s_1$ and $\lambda_2 s_2$ whose support is contained in $\Q_{\geq 0} \cdot S$.
    Let us address a few trivial cases before turning our attention to the generic setting.
    If $\lambda_1=\lambda_2=:\lambda$, then we can simply rescale the support of $\gamma$ by $\lambda$ while leaving the decoration monomials unchanged.
    The particular cases of $\lambda = 0$ and $\lambda = 1$ are the constant broken line segment with image the origin and the original broken line segment $\gamma$ respectively.
    Next, if $\lambda_i\neq \lambda_j =0$, we may take a straight segment between the origin and $\lambda_i s_i$.

    The remaining cases are less obvious, but follow from results of \cite{BLC}.
    Assume $\lambda_1$ and $\lambda_2$ are both non-zero.
    As in \cite{BLC}, denote the initial exponent vector of a broken line $\eta$ by $I(\eta)$ and the exponent vector of near the endpoint of $\eta$ by $\mf{m}_0(\eta)$.
    Define $\tau:= \frac{\lambda_2}{\lambda_1+\lambda_2} T$.
    Then the algorithm of \cite[\S5]{BLC} produces a balanced pair of broken lines $(\gamma^{(1)},\gamma^{(2)})$ where, for some $\mu>0$, 
    \begin{itemize}
        \item $I(\gamma^{(i)}) = \mu \lambda_i s_i$,
        \item $\mf{m}_0(\gamma^{(1)})+ \mf{m}_0(\gamma^{(2)}) = \mu (\lambda_1 + \lambda_2) \gamma(\tau)$, and
        \item $\gamma^{(1)}(0)=\gamma^{(2)}(0) = \mu (\lambda_1 + \lambda_2) \gamma(\tau)$.
    \end{itemize}
    Moreover, in this algorithm the bending points of the broken line segment $\gamma$ are positively proportional to the bending points of the pair $(\gamma^{(1)},\gamma^{(2)})$.
    (In the non-generic case in which $\gamma(\tau)$ is a bending point, the corresponding bend for the pair is by convention recorded in $\gamma^{(2)}$ in the algorithm.)
    
    Next, we take this pair of broken lines $(\gamma^{(1)},\gamma^{(2)})$, together with the pair of integers $a=b=1$, as input for the algorithm of \cite[\S4]{BLC}.
    The result is a broken line segment $\overline{\gamma}:[0,\overline{T}]\to U^{\trop}(\Q)$ with
    $\overline{\gamma}(0)= I(\gamma^{(1)})= \mu \lambda_1 s_1$ and $\overline{\gamma}(\overline{T})= I(\gamma^{(2)})= \mu \lambda_2 s_2$, passing through $\frac{\mu }{2}(\lambda_1 + \lambda_2) \gamma(\tau)$ at time $\frac{1}{2}\overline{T}$.
    As before, in this algorithm the bending points of the broken line segment $\overline{\gamma}$ are positively proportional to the bending points of the pair $(\gamma^{(1)},\gamma^{(2)})$, and thus positively proportional to the bending points of the broken line segment $\gamma$.
    Then the endpoints of each straight segment $\overline{L}$ of $\overline{\gamma}$ are positively proportional to the endpoints of the corresponding straight segment $L$ of $\gamma$.
    As a result, each such $\overline{L}$ is in $\Q_{\geq 0}\cdot L \subset \Q_{\geq 0}\cdot S$.
    
    Finally, we obtain the desired $\widetilde{\gamma}$ by rescaling the support (and elapsed time) of $\overline{\gamma}$ by $\frac{1}{\mu}$. 
\end{proof}

\begin{remark}
    Heuristically, in the proof of \thref{prop:WeakCone(S)}, we are translating between different tropical representations of the statement:
    \begin{center}
        $\tf_{\mu (\lambda_1+\lambda_2) \gamma(\tau)}$ is a non-zero summand of the product $\tf_{\mu \lambda_1 s_1} \tf_{\mu \lambda_2 s_2}$. 
    \end{center}
    In particular, if we consider the original input -- the broken line segment $\gamma$ and time $\tau$ -- the tropical point $\gamma(\tau)$ is viewed as a weighted average along $\gamma$ of the tropical points $\gamma(0)= s_1$ (with weight $\lambda_1$) and $\gamma(T)=s_2$ (with weight $\lambda_2$).
    Meanwhile, for the broken line segment $\overline{\gamma}$, the tropical point $\frac{\mu }{2}(\lambda_1 + \lambda_2) \gamma(\tau)$ is interpreted as the (unweighted) average along $\gamma$ of the tropical points $\mu \lambda_1 s_1$ and $\mu \lambda_2 s_2$. 
\end{remark}

\begin{remark}
    Note that if we were to consider a broken line convex set $S$ in \thref{prop:WeakCone(S)} rather than just a weakly convex set, then $\Q_{\geq}\cdot S$ would clearly be broken line convex.
    This follows immediately from the equivalence of broken line convexity and positivity (\cite[Theorem~6.1]{BLC}). 
\end{remark}

\begin{definition}\thlabel{def:WeakCone}
    If $S$ is weakly convex, we call $\Q_{\geq 0}\cdot S$ the {\it{weak cone of $S$}}.
    For arbitrary $S$, we call $\bconv(\Q_{\geq 0}\cdot S)$ the {\it{cone of $S$}} and denote it by $\Cone(S)$.
\end{definition}

\begin{definition}\thlabel{def:FaceFan}
    Let $S \subset U^{\trop}(\Q)$ be a full-dimensional polytopal set containing $0$ in the interior.
    The {\it{weak face fan of $S$}}, denoted $\Sigma[S]$, is the following collection of weak cones in $U^{\trop}(\Q)$:
    \eqn{\Sigma[S]:= \lrc{\sigma_F := \Q_{\geq 0} \cdot F : F \in \mathcal{F}_S } \cup \lrc{0}.}
\end{definition}

\begin{prop}\thlabel{prop:FaceFanPseudo-Complex}
    Let $S \subset U^{\trop}(\Q)$ be a full-dimensional polytopal set containing $0$ in the interior.
    Then the weak face fan of $S$ forms a pseudo-complex.
\end{prop}

\begin{proof}
    This follows almost immediately from \thref{prop:FacePseudo-Complex}.
    We will simply illustrate that Item~\ref{it:PBoundary} of \thref{def:pseudo-complex} holds for $\Sigma[S]$.
    The remaining items are recovered similarly.
    The weak cone $\lrc{0}$ has empty boundary, so there is nothing to do in this case.
    Now let $F\in \mathcal{F}_S$.  By \thref{prop:FacePseudo-Complex} there is a subset $\mathcal{A}_F$ of $\mathcal{F}_S$ such that
    \eqn{\partial F = \bigcup_{F' \in \mathcal{A}_F} F'.}
    Let $\mathcal{A}_{\sigma_F}:= \lrc{\sigma_{F'} \in \Sigma[S]\, : \,  F' \in \mathcal{A}_F}\cup\lrc{0}$.
    Then 
    \eqn{\partial \sigma_F = \bigcup_{\tau \in \mathcal{A}_{\sigma_F}} \tau.}
\end{proof}

We illustrate the weak face fan of the bigon from Figure~\ref{fig:Bigon1} in Figure~\ref{fig:Bigon1FaceFan} below.

\noindent
\begin{center}
\begin{minipage}{.85\linewidth}
\captionsetup{type=figure}
\begin{center}
\begin{tikzpicture}[scale=.85]

    \def\x{1.5}
    \def\l{3}
    \def\op{.3}

    \path (-\l,0) coordinate (3) --++ (\l,0) coordinate (0) --++ (\l,0) coordinate (1);
    \path (0,\l) coordinate (2) --++ (0,-2*\l) coordinate (4) --++ (\l,0) coordinate (5);
    \path (5) --++ (0,2*\l) coordinate (tr) --++ (-2*\l,0) coordinate (tl) --++ (0,-2*\l) coordinate (bl)--++ (2*\l,0) coordinate (br);

    \draw[thick, ->] (3) -- (1);
    \draw[thick, ->] (2) -- (4);
    \draw[thick, ->] (0) -- (5);

    \node at (.9,\l) {$1+z^{(0,1)}$};
    \node at (-2.3,-.35) {$1+z^{(-1,0)}$};
    \node at (3.25,-2) {$1+z^{(-1,1)}$};

    \coordinate (v1) at (-\x,\x);

    \path (v1) --++ (-.2*\x,.25*\x) node[color = more-blue] {$(-1,1)$};
    
    \coordinate (v2) at (\x,-\x);

    \path (v2) --++ (.5*\x,.2*\x) node[color = more-green] {$(1,-1)$};

    \coordinate (top-bend) at (0,0.5*\x);
    \coordinate (bottom-bend) at (-0.5*\x,0);
    
    \path [fill= blue-ish, fill opacity=\op] (v1.center)--(top-bend)--(v2.center)--(bottom-bend)--cycle;

    \draw[orange, thick] (v1.center)--(top-bend)--(v2.center);

    \draw[purple, thick] (v1.center)--(bottom-bend)--(v2.center);

    \node [circle, fill, inner sep = 1.5pt, color = more-blue] at (v1) {};

    \node [circle, fill, inner sep = 1.5pt, color = more-green] at (v2) {};
    
\begin{scope}[xshift=9cm]

    \path (-\l,0) coordinate (3) --++ (\l,0) coordinate (0) --++ (\l,0) coordinate (1);
    \path (0,\l) coordinate (2) --++ (0,-2*\l) coordinate (4) --++ (\l,0) coordinate (5);
    \path (5) --++ (0,2*\l) coordinate (tr) --++ (-2*\l,0) coordinate (tl) --++ (0,-2*\l) coordinate (bl)--++ (2*\l,0) coordinate (br);

    \draw[thick, ->, opacity=\op] (3) -- (1);
    \draw[thick, ->, opacity=\op] (2) -- (4);
    \draw[thick, ->, opacity=\op] (0) -- (5);

    \path[fill= orange, fill opacity=\op] (tl) -- (tr) -- (br) -- cycle;

    \path[fill= purple, fill opacity=\op] (tl) -- (bl) -- (br) -- cycle;
    
    \draw[ultra thick, more-blue] (0) -- (tl);

    \draw[ultra thick, more-green] (0) -- (br);

\end{scope}

\end{tikzpicture}

\captionof{figure}{\label{fig:Bigon1FaceFan}On the left, the bigon $\tc{blue-ish}{S}$ of Figure~\ref{fig:Bigon1}.
On the right, the weak face fan $\Sigma[\tc{blue-ish}{S}]$.
Note that for $F$ either facet, $\sigma_F$ is only weakly convex, not broken line convex.  In fact, $\bconv\lrp{\sigma_F}=(\cA^\vee)^{\trop}(\Q)$ for both facets.
} 
\end{center}
\end{minipage}
\end{center}

In order to generalize the convex geometry duality of \cite{Borisov} in 
future work,
it will be convenient to have a notion of support functions on weak face fans.

\begin{definition}\thlabel{def:SupportFunction}
Let $\Sigma[S]$ be the weak face fan of a full-dimensional polytopal set $S \subset U^{\trop}(\Q)$ containing $0$ in the interior.
A function $\varphi: U^{\trop}(\Q) \to \Q$ is a {\it{support function}} for $\Sigma[S]$ if for each weak cone $\sigma \in \Sigma[S]$ there is some $y_{\sigma}\in (U^\vee)^{\trop}(\Q)$ such that $\left.\varphi\right|_{\sigma} = \lra{\ \boldsymbol{\cdot}\ , y_\sigma}$.
A support function $\varphi$ is {\it{integral}} if each $y_\sigma$ may be taken to lie in $(U^\vee)^{\trop}(\Z)$.   
\end{definition}

\subsection{Duality for polyhedral sets and faces}

\begin{definition}
Let $S \subset U^\trop(\Q)$.
We define the {\it{polar of $S$}} to be
\eqn{S^\circ := \lrc{y \in (U^\vee)^\trop(\Q): \lra{s,y} \geq -1 \text{ for all } s \in S }.  }
More generally, for $r\in \Q_{\geq 0}$, we define the {\it{$r$-dual of $S$}} to be 
\eqn{S^{\vee_r} := \lrc{y \in (U^\vee)^\trop(\Q): \lra{s,y} \geq -r \text{ for all } s \in S }.  }
(In particular, if $r=1$, $S^{\vee_r}= S^\circ$.)
In the special case $r=0$, we simply write $S^\vee$ for $S^{\vee_0}$.
We define the polar and $r$-dual of subsets of $(U^\vee)^\trop(\Q)$ analogously.
\end{definition}

\begin{prop}\thlabel{prop:r-dual}
Let $S\subset U^{\trop}(\Q)$ and $r\in \Q_{\geq 0}$.  Then
    \eqn{S^{\vee_r} = \overline{\bconv\lrp{S\cup \lrc{0}}}^{\vee_r} .}
\end{prop}

\begin{proof}
    First note that $S \subset \overline{\bconv\lrp{S\cup \lrc{0}}}$, so the containment $\overline{\bconv\lrp{S\cup \lrc{0}}}^{\vee_r} \subset S^{\vee_r}$ is immediate.

    Next, let $y \in S^{\vee_r}$.
    That is, $\lra{x,y} \geq -r $ for all $x \in S$.
    Also, $\lra{0,y}= 0 \geq -r$.
    Then $\lrc{0} \cup S \subset K(y,r)$. 
    Since $K(y,r)$ is closed and broken line convex, this implies $\overline{\bconv\lrp{S\cup \lrc{0}}} \subset K(y,r)$.
    In other words, $y \in \overline{\bconv\lrp{S\cup \lrc{0}}}^{\vee_r} $, and $S^{\vee_r} = \overline{\bconv\lrp{S\cup \lrc{0}}}^{\vee_r}$.
\end{proof}

\begin{prop}\thlabel{prop:DoubleDual}
Let $S\subset U^{\trop}(\Q)$.  Then for $r>0$,
    \eqn{(S^{\vee_r})^{\vee_r} = \overline{\bconv\lrp{S\cup \lrc{0}}} .}
\end{prop}

\begin{proof}
    This is proved just like the classical version for polytopes in $\Q^n$.
    We follow the proof given in \cite[Theorem~6.2]{Brondsted}. 

    By \thref{lem:TropHalfSpaceBLC}, a tropical half-space is broken line convex.
    The $r$-dual of a set is by definition an intersection of closed tropical half-spaces, and the intersection of closed, broken line convex sets is closed and broken line convex.
    So, $(S^{\vee_r})^{\vee_r}$ is closed and broken line convex.
    Moreover, if $x\in S$ then $\lra{x,y}\geq -r $ for all $y$ in $S^{\vee_r}$ by definition of $S^{\vee_r}$.
    So $S \subset (S^{\vee_r})^{\vee_r}$, and obviously $\lrc{0} \subset (S^{\vee_r})^{\vee_r}$ as well.
    That is, $(S^{\vee_r})^{\vee_r}$ is a closed, broken line convex set containing $S \cup \lrc{0}$,
    and $(S^{\vee_r})^{\vee_r} \supset \overline{\bconv\lrp{S\cup \lrc{0}} }$.

    Next observe that
    \eqn{y \in S^{\vee_r}  \iff \lra{x, y} \geq -r \text{ for all } x \in S  \iff S \subset K(y,r).}
    So,
    \eqn{(S^{\vee_r})^{\vee_r} = \bigcap_{y \in S^{\vee_r}} K(y,r) = \bigcap_{K(y,r) \supset S }K(y,r). }
    Now take a point $z\notin \overline{\bconv\lrp{S \cup \lrc{0}}}$.
    There exists a supporting tropical half-space $K(y,r')$ of $\overline{\bconv\lrp{S \cup \lrc{0}}}$ with $z\notin K(y,r')$.
    So, 
    \eqn{ \min \lrc{\lra{x,y}: x\in \overline{\bconv\lrp{S\cup \lrc{0}}} } = -r' > \lra{z,y}. }
    Then there exists $t \in \Q_{>0}$ such that 
    \eqn{ \min \lrc{\lra{x,y}: x\in \overline{\bconv\lrp{S\cup \lrc{0}} }} \geq -t > \lra{z,y}. }
    Set $u:= \frac{r}{t} y$.
    Then 
    \eqn{ \min \lrc{\lra{x,u}: x\in \overline{\bconv\lrp{S\cup \lrc{0}} }} \geq -r > \lra{z,u}. }
    So $K(u,r) \supset S$, which implies $(S^{\vee_r})^{\vee_r} \subset K(u,r)$.
    But $z\notin K(u,r)$, so $z\notin (S^{\vee_r})^{\vee_r} $.
    That is, $z\notin \overline{\bconv\lrp{S \cup \lrc{0}}}$ implies $z\notin (S^{\vee_r})^{\vee_r} $.
    We conclude that \eqn{(S^{\vee_r})^{\vee_r} = \overline{\bconv\lrp{S\cup \lrc{0}}} .}
\end{proof}

\begin{lemma}\thlabel{lem:Halfspace-r-dual}
    Let $x\in U^{\trop}(\Q)$ be non-zero, and let $r\in \Q_{\geq 0}$.  Then $K(x,0)^{\vee_r} = \Q_{\geq 0}\cdot x$.
\end{lemma}

\begin{proof}
    First, $y\in K(x,0)$ if and only if $\lra{x,y}\geq 0$.
    For any $\lambda \geq 0$, we have $\lra{\lambda x,y}= \lambda \lra{x,y}$, and $\Q_{\geq 0}\cdot x \subset K(x,0)^{\vee} \subset K(x,0)^{\vee_r}$. 

    On the other hand, if $z\in K(x,0)^{\vee_r}$, then $\lra{z,y}\geq -r $ for all $y \in K(x,0)$.
    Suppose for some such $y$ we have $0>\lra{z,y}\geq -r$.
    Then for sufficiently large $\lambda>0$, we will have $\lra{z,\lambda y} < -r$.
    But $\lambda y \in K(x,0)$, so this contradicts the assumption that $z\in K(x,0)^{\vee_r}$.
    We find that in fact $K(x,0)^{\vee_r}=K(x,0)^{\vee}$.
    If $z$ is not a non-negative multiple of $x$, it will pair negatively with some $y\in H(x,0)$, so $K(x,0)^{\vee_r}=K(x,0)^{\vee} = \Q_{\geq 0} \cdot x$.
\end{proof}

\begin{prop}\thlabel{prop:r-dual-x-ray}
    Let $S\subset U^{\trop}(\Q)$ and $r \in \Q_{\geq 0}$.
    Then $S^{\vee_r} \subset K(x,0)$ if and only if $\Q_{\geq 0} \cdot x\subset \overline{\bconv\lrp{\lrc{0} \cup S}}$.
\end{prop}

\begin{proof}
    First observe that $(K(x,0)^{\vee_r})^{\vee_r} = K(x,0)$.
    This follows from \thref{lem:Halfspace-r-dual}. 
    Next, \thref{prop:r-dual} states that $S^{\vee_r} = \overline{\bconv\lrp{\lrc{0} \cup S}}^{\vee_r}$.
    So, we have $S^{\vee_r} \subset K(x,0)$ if and only if $\overline{\bconv\lrp{\lrc{0} \cup S}}^{\vee_r} \subset (K(x,0)^{\vee_r})^{\vee_r}$, which holds if and only if $K(x,0)^{\vee_r} \subset \overline{\bconv\lrp{\lrc{0} \cup S}}$.
    But \thref{lem:Halfspace-r-dual} states that $K(x,0)^{\vee_r}=\Q_{\geq 0} \cdot x$.
\end{proof}

\begin{prop}\thlabel{prop:DualCone}
    Let $S\subset U^\trop(\Q)$ be a weak cone.  Then $S^{\vee_r}$ is a cone.  Specifically, if $S=\Q_{\geq 0} \cdot T$ then $S^{\vee_r}= S^{\vee} = \Cone(T^\vee)$. 
\end{prop}

\begin{proof}
    Observe that 
\eqn{S^{\vee_r} = \lrp{\Q_{\geq 0}\cdot T}^{\vee_r}
= \bigcap_{x\in\Q_{\geq 0}\cdot T} K\lrp{x,r}
= \bigcap_{\lambda\in\Q_{\geq 0}} \bigcap_{t\in T} K\lrp{\lambda t,r}
= \lim_{\lambda\rightarrow\infty} \bigcap_{t\in T} K\lrp{\lambda t, r}.
}    
The last equality follows from the fact that if $\lambda_1\leq\lambda_2$, then $K(\lambda_1 t,r)\subset K(\lambda_2 t, r)$ for all $t\in T$. Later,
\eqn{\lim_{\lambda\rightarrow\infty} \bigcap_{t\in T} K\lrp{\lambda t, r} = \bigcap_{t\in T} K\lrp{t,0} 
= T^{\vee}
= \Q_{\geq 0}\cdot T^{\vee}
=\bconv\lrp{\Q_{\geq 0}\cdot T^{\vee}}
=\Cone\lrp{T^{\vee}}.
}
\end{proof}

\begin{prop}\thlabel{prop:CapDuals=DualCup}
Let $S$ and $T$ be subsets of $U^{\trop}(\Q)$.  Then
\eqn{\lrp{S^{\vee_r} \cap T^{\vee_r}}  = \lrp{S\cup T}^{\vee_r} .}
\end{prop}

\begin{proof}
    Let $y \in (U^\vee)^{\trop}(\Q)$.
    Then
    \eqn{y \in \lrp{S^{\vee_r} \cap T^{\vee_r}} \LRA \lra{x,y}\geq -r\  \text{ for all }x\in (S\cup T) \LRA y \in \lrp{S\cup T}^{\vee_r}.}
    
\end{proof}

\begin{prop}
Let $S$ and $T$ be subsets of $U^{\trop}(\Q)$ both containing $0$. Then
\eqn{\lrp{S \+t T}^{\vee_{r} }\subset \lrp{S^{\vee_r} \cap T^{\vee_r}}  \subset \lrp{S \+t T}^{\vee_{2r} }.}
\end{prop}

\begin{proof}
    Since both $S$ and $T$ contain $0$, we have that $S\subset S\+t T$ and $T\subset S\+t T$.
    This implies that $S^{\vee_r}\supset \lrp{S \+t T}^{\vee_r}$ and $T^{\vee_r}\supset \lrp{S \+t T}^{\vee_r}$, showing the first containment.
    
    Next, let $y\in S^{\vee_r}\cap T^{\vee_r}$. For an element $x\in S \+t T$, there exist $s\in S$ and $t\in T$ such that $x\in s\+t t$. Then
    \eqn{\lra{x,y}\geq \lra{s,y}+\lra{t,y}\geq -2r.}
    So, $y\in \lrp{S \+t T}^{\vee_{2r}}$.
\end{proof}

We obtain the following statement as a corollary:

\begin{cor}
Let $S$ and $T$ be subsets of $U^{\trop}(\Q)$ both containing $0$. Then
\eqn{\lrp{S^{\vee} \cap T^{\vee}}  = \lrp{S \+t T}^{\vee}.}
\end{cor}

\begin{prop}
Let $S$ and $T$ be closed cones in $U^{\trop}(\Q)$.
Then
\eqn{\lrp{S \cap T}^{\vee_r}  = S^{\vee_r} \+t T^{\vee_r}.}
\end{prop}

\begin{proof}
Since $S\cap T\subseteq S$ and $S\cap T\subseteq T$ we have that $\lrp{S\cap T}^{\vee_r}\supseteq S^{\vee_r}\cup T^{\vee_r}$. Consequently, 
\eqn{\bconv\lrp{S\cap T}^{\vee_r}\supseteq \bconv\lrp{S^{\vee_r}\cup T^{\vee_r}}.} 
Then  
\eqn{\lrp{S\cap T}^{\vee_r}=\bconv\lrp{S\cap T}^{\vee_r}\supseteq \bconv\lrp{S^{\vee_r}\cup T^{\vee_r}} = S^{\vee_r}\+t T^{\vee_r}.} 
Here, we use \thref{prop:r-dual} together with the fact that $S$ and $T$ are closed cones for the first equality.
For the second, \thref{prop:DualCone} implies $S^{\vee_r}=S^\vee$ and $T^{\vee_r}=T^\vee$ are closed cones as well, at which point we apply \thref{prop:ConvUnion-UnionSum}.

Conversely, since $S^{\vee_r}\subseteq S^{\vee_r}\cup T^{\vee_r}$, we have $\lrp{S^{\vee_r}}^{\vee_r}\supseteq \lrp{S^{\vee_r}\cup T^{\vee_r}}^{\vee_r}$. 
Then Proposition~\ref{prop:DoubleDual} implies the containment 
\eqn{S=\overline{\bconv\lrp{S\cup\{0\}}}\supseteq \lrp{S^{\vee_r}\cup T^{\vee_r}}^{\vee_r}.} 
Using an analogous argument, we have that $T\supseteq \lrp{S^{\vee_r}\cup T^{\vee_r}}^{\vee_r}$. 
So,  $(S\cap T)\supseteq \lrp{S^{\vee_r}\cup T^{\vee_r}}^{\vee_r}$, and in turn  $\lrp{S\cap T}^{\vee_r}\subseteq \lrp{\lrp{S^{\vee_r}\cup T^{\vee_r}}^{\vee_r}}^{\vee_r}$. 
Using once again Propositions~\ref{prop:DoubleDual}, \ref{prop:DualCone} and \ref{prop:ConvUnion-UnionSum}, we have that
\eqn{
\lrp{S\cap T}^{\vee_r}\subseteq \lrp{\lrp{S^{\vee_r}\cup T^{\vee_r}}^{\vee_r}}^{\vee_r}
&=\overline{\bconv\lrp{S^{\vee_r}\cup T^{\vee_r}\cup \{0\}}}\\
&=\bconv\lrp{S^{\vee_r}\cup T^{\vee_r}}\\
&=S^{\vee_r} \+t T^{\vee_r}.
}
This concludes the proof.
\end{proof}

\begin{prop}
Let $S$ and $T$ be closed broken line convex sets both containing 0. Then
\eqn{\overline{\bconv\lrp{S^{\circ} \cup T^{\circ}}}  = \lrp{S \cap T}^{\circ}.}
\end{prop}

\begin{proof}
Observe that by Proposition~\ref{prop:DoubleDual} we have that
\eqn{\overline{\bconv\lrp{S^{\circ}\cup T^{\circ}}}  = {\lrp{S^{\circ}\cup T^{\circ}}^{\circ}}^{\circ} 
= \lrp{{S^{\circ}}^{\circ}\cap {T^{\circ}}^{\circ}}^{\circ} 
= \lrp{\overline{\bconv\lrp{S\cup\{0\}}}\cap \overline{\bconv\lrp{T\cup\{0\}}}}^{\circ} 
= \lrp{S\cap T}^{\circ}.
}
\end{proof}

\begin{definition}
    A broken line convex set $S\subset U^{\trop}(\Q)$ is {\it{strongly}}
    broken line convex if no doubly infinite broken line has support contained in $S$.
\end{definition}

\begin{definition}
    Let $S \subset U^{\trop}(\Q)$, let $x$ be a non-zero element of $U^{\trop}(\Q)$, and fix a seed $\vb{s}$.
    We say that $S$ {\it{contains the asymptotic direction $x$}} if there exists a sequence $s_0,s_1,s_2, \dots$ of elements of $S$ such that 
    \eqn{\lim_{n \to \infty}\lVert \mf{r}_{\vb{s}}(s_{n}) \rVert = \infty \quad \text{and} \quad \lim_{n\to  \infty}\frac{\mf{r}_{\vb{s}}(s_{n}) }{\lVert \mf{r}_{\vb{s}}(s_{n}) \rVert} = \mf{r}_{\vb{s}}(x).}
\end{definition}

\begin{prop}\thlabel{prop:asymptotic-direction}
    $S$ contains the asymptotic direction $x$ if and only if $S^{\vee_r} \subset K(x,0)$.
\end{prop}

\begin{proof}
    We first show that if $\overline{\bconv\lrp{\lrc{0}\cup S}}$ contains the asymptotic direction $x$, then so does $S$.
    If we are given a sequence $\overline{s_0}, \overline{s_1}, \overline{s_2}, \dots$ in $\overline{\bconv\lrp{\lrc{0}\cup S}}$, we can obtain a sequence $s_0, s_1, s_2, \dots$ in $\bconv\lrp{\lrc{0}\cup S}$ with only small perturbations of each term.
    That is, for any $\epsilon>0$, we can ensure that each $s_n$ satisfies $\lVert \mf{r}_{\vb{s}}(s_n)- \mf{r}_{\vb{s}}(\overline{s_n}) \rVert <\epsilon$.
    Clearly, if 
    \eqn{\lim_{n \to \infty}\lVert \mf{r}_{\vb{s}}(\overline{s_{n}}) \rVert = \infty \quad \text{and} \quad \lim_{n\to  \infty}\frac{\mf{r}_{\vb{s}}(\overline{s_{n}}) }{\lVert \mf{r}_{\vb{s}}(\overline{s_{n}}) \rVert} = \mf{r}_{\vb{s}}(x),}
    then
    \eqn{\lim_{n \to \infty}\lVert \mf{r}_{\vb{s}}(s_{n}) \rVert = \infty \quad \text{and} \quad \lim_{n\to  \infty}\frac{\mf{r}_{\vb{s}}(s_{n}) }{\lVert \mf{r}_{\vb{s}}(s_{n}) \rVert} = \mf{r}_{\vb{s}}(x)}
    as well.
    Taking a subsequence if necessary, we may assume each $s_n$ is non-zero.
    Now we use \thref{prop:ConvUnion-UnionSum} to conclude that 
    \eqn{\bconv\lrp{\lrc{0}\cup S} = \bigcup_{a \in [0,1]} a S.}
    Take each $s_n$ to be in $a_n S$. Since we have assumed $s_n$ to be non-zero, $a_n \in (0,1]$.
    Now let $s_n' = \frac{1}{a_n} s_n$.
    Then we also have
    \eqn{\lim_{n \to \infty}\lVert \mf{r}_{\vb{s}}(s_{n}') \rVert = \infty \quad \text{and} \quad \lim_{n\to  \infty}\frac{\mf{r}_{\vb{s}}(s_{n}') }{\lVert \mf{r}_{\vb{s}}(s_{n}') \rVert} = \mf{r}_{\vb{s}}(x),}
    so $S$ contains the asymptotic direction $x$ if $\overline{\bconv\lrp{\lrc{0}\cup S}}$ does.
        
    By \thref{prop:r-dual-x-ray},  $S^{\vee_r} \subset K(x,0)$ if and only if $\Q_{\geq 0}\cdot x \subset \overline{\bconv\lrp{\lrc{0} \cup S}}$.
    But $\Q_{\geq 0}\cdot x $ clearly contains the asymptotic direction $x$.
    So if $S^{\vee_r} \subset K(x,0)$, then $S$ contains the asymptotic direction $x$.
    
    Similarly, if $S$ contains the asymptotic direction $x$, then so does $\overline{\bconv\lrp{\lrc{0}\cup S}}$.
    Then $\overline{\bconv\lrp{\lrc{0}\cup S}}$ is a closed, broken line convex set containing both $0$ and a sequence of points approaching the ray $\Q_{\geq 0}\cdot x$ at infinity.
    It must contain $\Q_{\geq 0}\cdot x$.
    Applying \thref{prop:r-dual-x-ray} again,
    we see that $S^{\vee_r} \subset K(x,0)$.
\end{proof}

\begin{prop}\thlabel{prop:SBLC-FullDim}
    Let $S \subset U^{\trop}(\Q)$ be broken line convex.  Then $S^{\vee_r} \subset (U^\vee)^{\trop}(\Q)$ is full-dimensional if and only if $S$ is {\emph{strongly}} broken line convex.
\end{prop}

\begin{proof}
    Suppose $S$ is {\emph{not}} strongly broken line convex.
    Then there exists a doubly infinite broken line $\gamma:\Q \to U^{\trop}(\Q)$ with support contained in $S$.
    We will choose a particular seed $\vb{s}$ to identify $ U^{\trop}(\Q)$ and $ (U^\vee)^{\trop}(\Q)$ with dual $\Q$-vector spaces $V$ and $V^*$ via the maps $\mf{r}_{\vb{s}}$ and $\mf{r}_{\vb{s}}^\vee$ as in \thref{not:Qvs}.
    Specifically, we choose $\vb{s}$ such that the support of $\gamma$ intersects the $\scat_{\vb{s}}^{U^\vee}$ chamber associated to $\vb{s}$ at a non-bending point $\mf{r}_{\vb{s}}(x_0)$ of $\mf{r}_{\vb{s}}(\supp(\gamma))$.
    Reparametrizing $\gamma$ if necessary, we can take $x_0 = \gamma(0)$.
    Now define
    \eqn{\eta_-: \Q_{\leq 0} &\to V\\
        t &\mapsto \mf{r}_{\vb{s}}(\gamma(t))}
    and
    \eqn{\eta_+: \Q_{\leq 0} &\to V\\
        t &\mapsto \mf{r}_{\vb{s}}(\gamma(-t)).}
    Since $\mf{r}_{\vb{s}}(x_0)$ is a non-bending point, for sufficiently small $\epsilon>0$, we have 
    \eq{\dot{\eta}_-(-\epsilon) = -\dot{\eta}_+(-\epsilon).}{eq:PatchingEtas}
    Write $v_{\pm}:= \lim_{t\to -\infty}\dot{\eta}_{\pm}(t)$.

    As $\eta_{\pm}$ has only finitely many bends, there exists some $R>0$ such that for all $t \in \Q_{\leq 0}$, $\eta_{\pm}(t)$ is contained in $B(R,t v_{\pm})$, the ball of radius $R$ centered at $t v_{\pm}$.
    That is, for all $t \in \Q_{\leq 0}$, we can write $\eta_{\pm}(t) = t v_{\pm} + b $ for some $b \in B(R,0)$.
    Now suppose $\lra{\mf{r}_{\vb{s}}(-v_{\pm}), y}<0$ for some $y\in (U^\vee)^{\trop}(\Q)$.
    Then 
    \eq{\lim_{t\to -\infty}\lra{\mf{r}_{\vb{s}}(t v_{\pm}), y}= -\infty.}{eq:pairing-infinity}

    The tropically linear function $\lra{\ \boldsymbol{\cdot} \ , y}: U^{\trop}(\Q) \to \Q$ defines a piecewise linear function on $V$ by 
    $\lrp{\mf{r}_{\vb{s}}^{-1}}^* \lra{\ \boldsymbol{\cdot} \ , y}$,
    and this piecewise linear function has the form 
    \eqn{\lrp{\mf{r}_{\vb{s}}^{-1}}^* \lra{\ \boldsymbol{\cdot} \ , y}= \min_{\ell \in L} \lrc{\ell( \ \boldsymbol{\cdot} \ )}}
    for some finite set $L$ of linear functions on $V$.
    Then 
    \eqn{\lra{\mf{r}_{\vb{s}}^{-1}\lrp{\eta_{\pm}(t)},y} = \min_{\ell \in L} \lrc{\ell(\eta_{\pm}(t))} =  \min_{\ell \in L} \lrc{\ell(t v_{\pm}) + \ell(b)}  }
    for some $b\in B(R,0)$.
    However, $\left.\ell\right|_{B(R,0)}$ is bounded for all $\ell \in L$, while $\lim_{t\to -\infty} \min_{\ell \in L}\ell(t v_{\pm}) = -\infty$ by \eqref{eq:pairing-infinity}.
    So, $ \lim_{t\to -\infty} \lra{\mf{r}_{\vb{s}}^{-1}\lrp{\eta_{\pm}(t)},y} = -\infty$
    as well, and $y\notin \supp(\gamma)^{\vee_r}$.
    In other words, if  $y\in \supp(\gamma)^{\vee_r}$, then $\lra{\mf{r}_{\vb{s}}^{-1}(v_{\pm}), y} \geq 0$.

    Next, 
    \eqn{\lra{\mf{r}_{\vb{s}}^{-1}(v_{\pm}), y} = \tf_{\mf{r}_{\vb{s}}^{-1}(v_{\pm})}^{\trop}(y) &= \lrp{\mf{r}_{\vb{s}}^\vee(y)}\lrp{\tf_{v_{\pm}, \mf{r}_{\vb{s}}(x_0)}}\\
    &= \min\lrc{m(\mf{r}_{\vb{s}}^\vee(y)) \,: \, z^m \text{ is a non-zero summand of } \tf_{v_{\pm}, \mf{r}_{\vb{s}}(x_0)}}\\
    &\leq \lrp{- \dot{\eta}_{\pm}(-\epsilon)}\lrp{\mf{r}_{\vb{s}}^\vee(y)} \text{ for small }\epsilon>0.}
    That is, if $y\in \supp(\gamma)^{\vee_r}$, then $0 \leq \lra{\mf{r}_{\vb{s}}^{-1}(v_{\pm}), y} \leq \lrp{- \dot{\eta}_{\pm}(-\epsilon)}\lrp{\mf{r}_{\vb{s}}^\vee(y)} $ for small $\epsilon>0$.
    Then \eqref{eq:PatchingEtas} implies $\mf{r}_{\vb{s}}^\vee(y) \in \dot{\eta}_{\pm}(-\epsilon)^\perp$, and $\supp(\gamma)^{\vee_r}$ is not full-dimensional.
    But $\supp(\gamma) \subset S$, so $S^{\vee_r} \subset \supp(\gamma)^{\vee_r}$, and $S^{\vee_r}$ is also not full-dimensional.

    Now suppose $S^{\vee_r}$ is not full-dimensional.
    Say $d:= \dim(S^{\vee_r})$.
    Choose a seed $\vb{s}$ such that the chamber $\sigma_{\vb{s}}$ of $\scat_{\vb{s}}^{U}$ associated to the seed $\vb{s}$ intersects $\mf{r}_{\vb{s}}^\vee(S^{\vee_r})$ in a $d$-dimensional subset.
    Note that $\mf{r}_{\vb{s}}^\vee(S^{\vee_r})$ is contained in some hyperplane through the origin, say $m^{\perp}$ for some integral $m$.
    Let $x_+$ and $x_-$ to be the points of $U^{\trop}(\Z)$ with 
    \eq{\left.\lrp{{\mf{r}_{\vb{s}}^\vee}^{-1}}^*\lra{x_{\pm},\ \boldsymbol{\cdot} \ }\right|_{\sigma_{\vb{s}}} = \pm m (\ \boldsymbol{\cdot} \ ).}{eq:xpm}
    Then clearly 
    \eqn{ S^{\vee_r} \cap (\mf{r}_{\vb{s}}^\vee)^{-1}(\sigma_{\vb{s}}) \subset \lrp{ H(x_+,0) \cap H(x_-,0)}.}
    Moreover, since $S^{\vee_r}$ is broken line convex, \thref{prop:BLInHyperplanes} implies that in fact 
    \eq{ S^{\vee_r} \subset \lrp{ H(x_+,0) \cap H(x_-,0)}.}{eq:H+capH-}
    
    Note that \eqref{eq:xpm} implies there is a pair of broken lines $\eta_{-}, \eta_+$ in $V$ with initial exponent vectors $\mf{r}_{\vb{s}}(x_-), \mf{r}_{\vb{s}}(x_+)$
    and final exponent vectors $-m, +m$
    sharing the same basepoint.
    This pair of broken lines indicates that the product $\tf_{x_-} \tf_{x_+}$ has non-zero constant ($\tf_{0}$) term.
    Explicitly, we can dilate the supports of the pair of broken lines to bring the basepoint arbitrarily close to the origin.
    Then this pair of broken lines precisely describes a contribution to the product $\tf_{x_-} \tf_{x_+}$ as described in \cite[Definition-Lemma~6.2]{GHKK}.
    However, by \cite[Proposition~6.4.(3)]{GHKK}, we can compute the structure constant $\alpha_{x_+,x_-}^{0}$ of this multiplication using any basepoint near the origin.
    In particular, we may choose a basepoint $\mf{r}_{\vb{s}}(x_b)$ such that $\lambda x_b$  is in the relative interior of $S$ for some $\lambda> 0$.
    Then we obtain a pair of broken lines in $V$ with basepoint $\mf{r}_{\vb{s}}(x_b)$, initial exponent vectors $\mf{r}_{\vb{s}}(x_-)$ and $\mf{r}_{\vb{s}}(x_+)$, and final exponent vectors summing to $0$.
    Dilating the supports of these broken lines by $\lambda$, we obtain a such a pair with basepoint in the relative interior of $S$.
    We may reverse the direction of one of the broken lines to obtain a doubly infinite broken line $\gamma$ passing through the previous basepoint and having 
    \eq{\lim_{t\to \pm \infty} \dot{\gamma}(t) = -\mf{r}_{\vb{s}}(x_\pm).}{eq:gamma-dot}
    
    Next, \eqref{eq:H+capH-} and \thref{prop:asymptotic-direction} together imply $S$ contains the asymptotic directions $x_+$ and $x_-$.
    Now suppose $\supp(\gamma) \not\subset \mf{r}_{\vb{s}}(S)$.
    Then at some point $\gamma$ must leave $\mf{r}_{\vb{s}}(S)$.
    As $\gamma$ passes through the relative interior of $\mf{r}_{\vb{s}}(S)$, \thref{prop:BLInHyperplanes} prevents $\gamma$ from simply entering and remaining in the boundary of the closure of $\mf{r}_{\vb{s}}(S)$ in the event that $\mf{r}_{\vb{s}}(S)$ is not closed.
    Then points of $\gamma$ must eventually be a positive distance from $\mf{r}_{\vb{s}}(S)$.
    However, since $S$ contains the asymptotic directions $x_\pm$, \eqref{eq:gamma-dot} implies this positive distance is bounded.
    Then as argued in \thref{prop:BLInHyperplanes},
    we may take a small perturbation $\gamma'$ of a segment of $\gamma$, this time adding a small contribution to the velocity at $\lambda x_b$ so that the first bend for $\gamma'$ after leaving $\mf{r}_{\vb{s}}(S)$ is slightly closer to $\mf{r}_{\vb{s}}(S)$ the corresponding bend of $\gamma$.
    Keeping all wall contributions the same (up to a multiplicative constant to maintain integrality of exponent vectors) as in \thref{prop:BLInHyperplanes},
    we obtain a broken line segment which must eventually re-enter $\mf{r}_{\vb{s}}(S)$ as the direction after the last bend will have a small contribution directed toward $\mf{r}_{\vb{s}}(S)$, much like the situation illustrated in Figure~\ref{fig:ContradictConvexity}.
    This contradicts the assumption that $S$ is broken line convex.
    As a result, we conclude that $\mf{r}_{\vb{s}}^{-1}(\supp(\gamma))$ is contained in $S$ and so $S$ is not strongly broken line convex.
\end{proof}

We have treated $r$-duals thus far so that we can apply our results equally well in the two main cases of interest: $r=0$ and $r=1$.
For the remainder of this subsection, we will focus on questions that are primarily interesting for polar duality.
With this in mind, we will return to the simpler ${}^\circ$ notation and comment that our arguments may easily be adapted to treat the more general $r>0$ case. 

\begin{prop}\thlabel{prop:Bounded-FullDim0}
    Let $S \subset U^{\trop}(\Q)$ be broken line convex. Then $S^\circ \subset (U^\vee)^{\trop}(\Q)$ is bounded if and only if $S$ is full-dimensional and contains the origin in its interior.
\end{prop}

\begin{proof}
    Suppose $S$ is full-dimensional and contains the origin in its interior.
    Choose a seed $\vb{s}$ to identify $U^{\trop}(\Q)$ and $(U^\vee)^{\trop}(\Q)$ with dual $\Q$-vector spaces $V$ and $V^*$ via $\mf{r}_{\vb{s}}$ and $\mf{r}_{\vb{s}}^\vee$ as in \thref{not:Qvs}.
    Then $\mf{r}_{\vb{s}}(S)$ contains the ball $B(R,0)$ for some sufficiently small $R>0$.
    For any $x\in U^{\trop}(\Q)$, 
    we have
    \eqn{\lrp{{\mf{r}_{\vb{s}}^\vee}^{-1}}^*\lra{x, \ \boldsymbol{\cdot}\ } = \min_{\ell \in L}\lrc{\ell(\  \boldsymbol{\cdot}\ )}}
    for some finite set of linear maps $L$ containing $\mf{r}_{\vb{s}}(x)$.
    Then 
    \eqn{\mf{r}_{\vb{s}}^\vee\lrp{K(x,1)} \subset K(\mf{r}_{\vb{s}}(x),1),} where $K(\mf{r}_{\vb{s}}(x),1):= \lrc{v\in V^*\, :\, \lrp{\mf{r}_{\vb{s}}(x)}(v) \geq -1}$.
    
    Now note that the polar of $B(R,0) \subset V$ is $B\lrp{\frac{1}{R},0} \subset V^*$.
    Then we have
    \eqn{\mf{r}_{\vb{s}}^\vee\lrp{S^\circ}=\mf{r}_{\vb{s}}^\vee\lrp{\bigcap_{s\in S}K(s,1)} \subset \mf{r}_{\vb{s}}^\vee\lrp{\bigcap_{\mf{r}_{\vb{s}}(s)\in B(R,0)} K(s,1)} \subset B(R,0)^\circ = B\lrp{\frac{1}{R},0},} 
    and $S^\circ$ is bounded.

    Next, if $S$ is not full-dimensional, $S^\circ$ is not strongly convex by \thref{prop:SBLC-FullDim}, and {\it{a fortiori}} not bounded.
    If $S$ does not contain the origin in the interior, then it is contained in $K(y,0)$ for some non-zero ${y\in (U^\vee)^{\trop}(\Q)}$.
    So $\Q_{\geq 0} \cdot y \in S^\circ$ and $S^\circ$ is not bounded.
\end{proof} 

We have the following immediate corollary:

\begin{cor}\thlabel{cor:BoundedFullDim}
    If $S$ is a bounded, full-dimensional set containing the origin in its interior, then so is $S^\circ$.
\end{cor}

\begin{definition}
Let $S \subset U^\trop(\Q)$.
We say $S$ is {\it{integral}} if $S= \bconv\lrp{A}$ for some finite subset $A$ of $ U^\trop(\Z)$.  
We define integral subsets of $(U^\vee)^\trop(\Q)$ analogously, and
we say $S$ is {\it{reflexive}} if $0\in S$ and both $S$ and $S^\circ$ are integral.
\end{definition}

\begin{remark}
    
\end{remark}

\begin{definition}\thlabel{def:FaceDual}
    Let $S \subset U^{\trop}(\Q)$ be a full-dimensional polytopal subset and containing $0$ in the interior.
    We define the {\it{dual of a face}} $F$ of $S$ to be
   \eqn{\widecheck{F}:=\lrc{y \in S^\circ \, : \, \lra{x,y} = -1 {\text{ for all } x \in F}}.}
\end{definition}

Note that for the empty face $\varnothing$, we have $\widecheck{\varnothing}= S^\circ$.

\begin{prop}\thlabel{prop:FaceDual}
    Let $S \subset U^{\trop}(\Q)$ be a full-dimensional polytopal subset and containing $0$ in the interior,
    and let $F$ be a proper face of $S$.
    Then $\widecheck{F}$ is a proper face of $S^\circ$. 
    Precisely, if $x$ is in the relative interior of $F$, then $\widecheck{F} = S^\circ \cap H(x,1)$.
    Moreover, if $F'\subsetneq F$ then $\widecheck{F}\subsetneq \widecheck{F}'$ and $\widecheck{\widecheck{F}}=F$. 
    This gives a bijective, containment-reversing correspondence between proper faces of $S$ and $S^\circ$.
\end{prop}

\begin{proof}

    For any $x \in F$, define $F_x:=S^\circ \cap H(x,1) $.
    Observe that 
    \eqn{\widecheck{F}= \bigcap_{x\in F} F_x,}
    and $\widecheck{F} \subset F_x$ for all $x\in F$.
    Now suppose $x$ is in the relative interior of $F$.
    Then by \thref{def:pseudo-complex} Item~\ref{it:PIntersectionInBoundary} and \thref{prop:FacePseudo-Complex},
    a supporting tropical hyperplane $H(y,r)$ for $S$ which contains $x$  must in fact contain $F$. 
    So if $y\in F_x$, then $F \subset H(y,1)$, and $y \in \widecheck{F}$.
    That is, $\widecheck{F}=F_x$ for any $x$ in the relative interior of $F$.

    Next, we claim that every proper face of $S^\circ$ is the dual of a face of $S$.
    By definition, every face of $S^\circ$ is of the form $S^\circ \cap H(x,r)$ for some $x\in U^{\trop}(\Q)$, and using \thref{rem:Rescale} we can choose $x$ such that $r=1$.
    But such an $x$ is necessarily in the boundary of $S$ since $(S^\circ)^\circ =S$.
    Every boundary point is contained in some face, and moreover by \thref{prop:FacePseudo-Complex}, contained in the relative interior of a face. Then the previous argument implies that every proper face of $S^\circ$ is the dual of a face of $S$. 
    
    It follows immediately from the definition of the dual of a face that  $F' \subset F$ implies $\widecheck{F}\subset \widecheck{F}'$.
    If moreover $F' \subsetneq F$, then 
    there is some $y' \in (U^\vee)^{\trop}(\Q)$ and $r' \in \Q_{>0}$ with $F'= S \cap H(y',r')$ and $F \subset S \subset K(y',r')$.
    As argued above, we may take $r'=1$ and $y' \in S^{\circ}$.
    Then $y'\in \widecheck{F}' \setminus \widecheck{F}$, so $\widecheck{F}\subsetneq \widecheck{F}'$.
   
    Finally, $S$ and $S^\circ$ play completely interchangeable roles here.
    So, every proper face of $S$ is the dual of a proper face of $S^\circ$ as well.
    We have automatically that $F \subset \widecheck{\widecheck{F}}$.
    Suppose  $F \subsetneq \widecheck{\widecheck{F}}$.
    Since proper containments are reversed by duality of faces, we must also have $\widecheck{\widecheck{\widecheck{F}}} \subsetneq \widecheck{F}$.
    But this violates the automatic containment $\widecheck{F} \subset \widecheck{\widecheck{\widecheck{F}}}$.
    We conclude that $F = \widecheck{\widecheck{F}}$.
    
\end{proof}

See Figure~\ref{fig:BigonsDualFacePseudoComplex} for an example of this duality for face pseudo-complexes.

\noindent
\begin{center}
\begin{minipage}{.85\linewidth}
\captionsetup{type=figure}
\begin{center}
\begin{tikzpicture}[scale=.7]

    \def\x{1.3}
    \def\l{4}
    \def\op{.3}

    \path (-\l,0) coordinate (3) --++ (\l,0) coordinate (0) --++ (\l,0) coordinate (1);
    \path (0,\l) coordinate (2) --++ (0,-2*\l) coordinate (4) --++ (\l,0) coordinate (5);
    \path (5) --++ (0,2*\l) coordinate (tr) --++ (-2*\l,0) coordinate (tl) --++ (0,-2*\l) coordinate (bl)--++ (2*\l,0) coordinate (br);

    \draw[thick, ->, opacity = \op] (3) -- (1);
    \draw[thick, ->, opacity = \op] (2) -- (4);
    \draw[thick, ->, opacity = \op] (0) -- (5) node [pos=.7, sloped, above] {$1+z^{e_2^*-e_1^*}$};

    \node [opacity = \op] at (.75,\l) {$1+z^{e_2^*}$};
    \node [opacity = \op] at (-3,-.35) {$1+z^{-e_1^*}$};

    \coordinate (v1) at (-\x,\x);

    \path (v1) --++ (-.2*\x,.25*\x) node[color = more-blue] {$v_1$};
    
    \coordinate (v2) at (\x,-\x);

    \path (v2) --++ (.3*\x,.2*\x) node[color = more-green] {$v_2$};

    \coordinate (top-bend) at (0,0.5*\x);
    \coordinate (bottom-bend) at (-0.5*\x,0);
    
    \path [fill= blue-ish, fill opacity=\op] (v1.center)--(top-bend)--(v2.center)--(bottom-bend)--cycle;

    \draw[orange, thick] (v1.center)--(top-bend)--(v2.center);
    
    \path (top-bend) --++ (.4*\x,0) node [color=orange] {$F$};
 
    \draw[purple, thick] (v1.center)--(bottom-bend)--(v2.center);

    \path (bottom-bend) --++ (0,-.4*\x) node [color=purple] {$F'$};

    \node [circle, fill, inner sep = 1.5pt, color = more-blue] at (v1) {};

    \node [circle, fill, inner sep = 1.5pt, color = more-green] at (v2) {};
 
\begin{scope}[xshift=10cm]

    \def\Honeonebot{3}
    \def\Honeonetop{3}
    \def\Htwoonebot{4}
    \def\Htwoonetop{3}

    \path (-\l,0) coordinate (3) --++ (\l,0) coordinate (0) --++ (\l,0) coordinate (1);
    \path (0,\l) coordinate (2) --++ (0,-2*\l) coordinate (4) --++ (-\l,0) coordinate (5);
    \path (5) --++ (0,2*\l) coordinate (tl) --++ (2*\l,0) coordinate (tr) --++ (0,-2*\l) coordinate (br) --++ (-2*\l,0) coordinate (bl);

    \draw[thick, ->, opacity=\op] (1) -- (3);
    \draw[thick, ->, opacity=\op] (2) -- (4);
    \draw[thick, ->, opacity = \op] (0) -- (5) node [pos=.7, sloped, below] {$1+z^{e_1+e_2}$};

    \node [opacity = \op] at (.75,\l) {$1+z^{e_2}$};
    \node [opacity = \op] at (3,-.35) {$1+z^{e_1}$};

    \path[name path = left] (tl) -- (bl);    
    \path[name path = right] (tr) -- (br);
    \path[name path = top] (tl) -- (tr);
    \path[name path = bottom] (bl) -- (br);

    \coordinate (bend11) at (\x,0);

    \path[name path = path11bot] (bend11)--++ (-2*\Honeonebot, -\Honeonebot);

    \path[name intersections={of=left and path11bot, by=path11end}];

    \path[name path = path11top] (bend11)--++ (\Honeonetop, \Honeonetop);

    \path[name intersections={of=right and path11top, by=path11start}];

    \path[name path = path11] (path11start)--(bend11)--(path11end);

    \coordinate (bend21) at (0,\x);

    \path[name path = path21bot] (bend21)--++ (-\Htwoonebot, -\Htwoonebot);

    \path[name intersections={of=path11 and path21bot, by=vbl}];

    \path[name path = path21top] (bend21)--++ (\l, 0);

    \path[name intersections={of=path11 and path21top, by=vtr}];

    \path [fill=violet, fill opacity=\op] (vbl.center)--(bend11)--(vtr.center)--(bend21)--cycle;

    \draw[very thick, color=more-green] (vbl)--(bend21)--(vtr);

    \draw[very thick, color=more-blue] (vbl)--(bend11)--(vtr);

    \node[circle, fill, inner sep = 1.5pt, color = orange] at (vbl) {};
    \node[circle, fill, inner sep = 1.5pt, color = purple] at (vtr) {};
    \path (bend11) --++ (-.5*\x,-.6*\x) node [color = more-blue] {$\widecheck{v_1}$};
    \path (bend21) --++ (-\x,-.5*\x) node [color = more-green] {$\widecheck{v_2}$};
    \path (vbl) -++ (0,-.4*\x) node [color = orange] {$\widecheck{F}$};
    \path (vtr) -++ (.5*\x,.25*\x) node [color = purple] {$\widecheck{F}'$};
\end{scope}

\end{tikzpicture}

\captionof{figure}{\label{fig:BigonsDualFacePseudoComplex}On the left, the bigon $\tc{blue-ish}{S}$ of Figure~\ref{fig:Bigon1} together with its face pseudo-complex.
On the right, the dual face pseudo-complex of $\tc{violet}{S^\circ}$, which is also a bigon.
} 
\end{center}
\end{minipage}
\end{center}


For future work generalizing Borisov's duality for nef-partitions (see \cite{Borisov} for the original), it will be useful to describe some duality results for sets cut out by functions on $U^{\trop}(\Q)$.  

\begin{prop}\thlabel{prop:DualXi}
    Consider a set $T\subset (U^{\vee})^{\trop}(\Q)$, and define
    \eqn{\varphi:U^{\trop}(\Q) &\to \Q\\
            x&\mapsto \min_{y\in T}\lrc{\lra{x,y}}.    
    }
    Then 
    \eqn{{\Xi_{\varphi,1}} =  \lrc{y\in (U^\vee)^{\trop}(\Q)\, :\, \lra{x,y} \geq \varphi(x) \text{ for all } x\in U^{\trop}(\Q)}^\circ = T^\circ.}
\end{prop}

\begin{proof}
For the first equality, we compute:
    \eq{{\Xi_{\varphi,1}}^\circ:=& \lrc{y\in (U^\vee)^{\trop}(\Q)\, :\, \lra{x,y} \geq -1 \text{ for all } x\in \Xi_{\varphi,1}}\\
    =& \lrc{y\in (U^\vee)^{\trop}(\Q)\, :\, \lra{x,y} \geq -1 \text{ for all } x\in \partial \Xi_{\varphi,1}}\\
    =& \lrc{y\in (U^\vee)^{\trop}(\Q)\, :\, \lra{x,y} \geq \varphi(x) \text{ for all } x\in \partial \Xi_{\varphi,1}}\\
    =& \lrc{y\in (U^\vee)^{\trop}(\Q)\, :\, \lra{\lambda x,y} \geq \varphi(\lambda x) \text{ for all } x\in \partial \Xi_{\varphi,1},\ \lambda \in \Q_{\geq 0}}\\
    =& \lrc{y\in (U^\vee)^{\trop}(\Q)\, :\, \lra{x,y} \geq \varphi(x) \text{ for all } x\in U^{\trop}(\Q)}.
    }{eq:XiPhi1}
But $\Xi_{\varphi,1}$ is a closed broken line convex set containing $0$, so \thref{prop:DoubleDual} implies the first equality.

 For the second equality, suppose that $x\in {T}^{\circ}$. Then $\lra{x,y}\geq -1$ for all $y\in T$ and thus
    \eqn{\varphi(x) = \min_{y'\in T}\lrc{\lra{x,y'}} \geq -1 .}
So, $x\in\Xi_{\varphi,1}$ and this implies that ${T}^{\circ}\subseteq\Xi_{\varphi,1}$. Conversely, suppose $y\in T$. Then for all $x \in U^{\trop}(\Q)$,
    \eqn{\lra{x,y} \geq \min_{y'\in T}\lrc{\lra{x,y'}} = \varphi(x).}
So, $y\in{\Xi_{\varphi,1}}^{\circ}$ by Equation~\eqref{eq:XiPhi1}. 
We have then ${T\subseteq{\Xi_{\varphi,1}}^{\circ}}$ and consequently $T^{\circ}\supseteq \Xi_{\varphi,1}$.   
\end{proof}

\begin{prop}\thlabel{prop:DualPolysOfSupportFuns}
    Let $S \subset U^{\trop}(\Q)$ be a full-dimensional polytopal set containing $0$ in the interior, and let $\varphi:U^{\trop}(\Q) \to \Q$ be a support function for $\Sigma[S]$ which is convex with respect to broken lines.     Then
    $\Xi_{\varphi,1}\subset U^{\trop}(\Q)$ is a full-dimensional polyhedral set containing $0$ in the interior.
    Meanwhile, ${\Xi_{\varphi,1}}^\circ\subset (U^\vee)^{\trop}(\Q)$ is bounded and given by
    \eqn{{\Xi_{\varphi,1}}^\circ &=
    \bconv\lrp{\lrc{0}\cup \bigcup_{\sigma \in \Sigma[S]}\lrc{y_\sigma}},}
    where $y_\sigma$ satisfies $\left.\varphi\right|_\sigma = \lra{\ \boldsymbol{\cdot}\ ,y_\sigma}$.
\end{prop}

\begin{proof}
    Since $\varphi$ is a support function for $\Sigma[S]$, for all $\sigma \in \Sigma[S]$ we have $\left.\varphi\right|_\sigma = \lra{\ \boldsymbol{\cdot}\ ,y_\sigma}$ for some tropical point $y_\sigma \in (U^\vee)^\trop(\Q)$, and
    \eqn{\Xi_{\varphi,1}\cap \sigma = \lrc{x \in \sigma \, : \, \lra{x,y_\sigma} \geq -1} = K(y_\sigma,1)\cap \sigma.}
    Next, $\Xi_{\varphi,1}$ is broken line convex by \thref{prop:ConvWRTBL-BLC}, and given by
    \eqn{\Xi_{\varphi,1} = \bigcap_{\sigma \in \Sigma[S]}K(y_{\sigma},1).}
    This shows that $\Xi_{\varphi,1}$ is polyhedral.
    As the indexing set of the intersection is finite and each $K(y_\sigma,1)$ is a full-dimensional set containing $0$ in the interior, so is $\Xi_{\varphi,1}$. 
    By \thref{prop:Bounded-FullDim0}, ${\Xi_{\varphi,1}}^\circ$ is bounded.
    Moreover, $K(y_{\sigma},1) = \lrc{y_\sigma}^\circ$, so using \thref{prop:CapDuals=DualCup} and \thref{prop:DoubleDual}
    we see that 
    \eqn{\Xi_{\varphi,1} = \bigcap_{\sigma \in \Sigma[S]}\lrc{y_\sigma}^\circ =  \lrp{\bigcup_{\sigma \in \Sigma[S]}\lrc{y_\sigma}}^\circ}
    and
    \eqn{{\Xi_{\varphi,1}}^\circ=  \lrp{\lrp{\bigcup_{\sigma \in \Sigma[S]}\lrc{y_\sigma}}^\circ}^\circ
    = \bconv\lrp{\lrc{0}\cup \bigcup_{\sigma \in \Sigma[S]}\lrc{y_\sigma}}.}
\end{proof}

\subsection{Half-space and vertex representations}

\begin{prop}\thlabel{prop:ConvFinite-ConvVert}
    Let $A\subset U^{\trop}(\Q)$ be a finite collection of points and let $S=\bconv(A)$.
    Then necessarily $V(S) \subset A$ and $S=\bconv(V(S))$.
\end{prop}

\begin{proof}
    Let $v\in A$, but $v\notin \bconv(A\setminus\lrc{v})=:S_v$.
    Then there is some supporting tropical half-space $K(y,r)$ for $S_v$ with $x\notin K(y,r)$.
    For all $r'>r$, we have that $S_v$ is contained in the interior of $ K(y,r')$.  Moreover, for some such $r'$, we have that $v \in H(y,r')$. 

    Now let $s\in S$.
    By \thref{prop:ConvUnion-UnionSum},
    \eqn{s \in \Sumt_{x\in A} a_x x}
    for some collection of non-negative $a_x$ which sum to $1$.
    Next, \thref{prop:ConvWRTBL} implies
    \eqn{\lra{s,y} &\geq  \sum_{x\in A} a_x \lra{x,y}\\
    &\geq -  a_v r' - \sum_{x \in A\setminus\lrc{v}} a_x r\\
    &\geq - r',    }
    with equality if and only if $a_v=1$, $a_{x\neq v} = 0$.
    That is, if $s\in S$ is in $H(y,r')$, then $s=v$.  Thus $v$ is a vertex of $S$.    
\end{proof}

\begin{prop}\thlabel{prop:PolytopeConvFinite}
    Every polytopal set $S\subset U^{\trop}(\Q)$ is the broken line convex hull of a finite set.
\end{prop}

\begin{proof}
    This follows directly from the usual convex geometry statement by choosing a seed-- the only subtlety being that the finite set of points obtained in this way will generally not be minimal.
    
    Let $d = \dim(U)$.
    Then a choice of seed identifies $U^{\trop}(\Q)$ with $\Q^d$, and $S$ with a polytope $P_S$ in $\Q^d$.
    Specifically, each tropical half-space defining $S$ is identified with an intersection of a finite number of half-spaces in $\Q^d$.
    Thus $P_S$ is the intersection of finitely many half-spaces in $\Q^d$; it is a rational polyhedron, and in fact a polytope as it is bounded.
    Then $P_S$ is the (usual) convex hull of a finite collection of points in $\Q^d$, namely the vertices of $P_S$ in the usual convex geometry sense.
    Denote by $A$ the collection of tropical points associated to this set of vertices of $P_S$.
    The equality $S = \bconv(A)$ is clear.
\end{proof}

We can combine Propositions~\ref{prop:ConvFinite-ConvVert} and \ref{prop:PolytopeConvFinite} to obtain
\begin{cor}\thlabel{cor:Hrep-Vrep}
    If $S\subset U^{\trop}(\Q)$ is polytopal, then $V(S)$ is finite and $S=\bconv(V(S))$. 
\end{cor}

So, we have defined polytopal sets using a broken line convex geometry version of the ``half-space representation''.
\thref{cor:Hrep-Vrep} indicates that a polytopal set always has a broken line convex geometry version of the ``vertex representation'' as well.
We may also perform this translation in the opposite direction:

\begin{prop}
    Let $A\subset U^{\trop}(\Q)$ be a finite set such that $S:=\bconv(A)$ is full-dimensional and contains the origin.  Then $S$ is polytopal.
\end{prop}

\begin{proof}
    First note that $S^\circ$ is given by
    \eqn{S^\circ = \bigcap_{v\in V(S)} K(v,1). }
    By \thref{prop:ConvFinite-ConvVert} $V(S)$ is a finite set contained in $A$, and so $S^\circ$ is polyhedral.
    Using \thref{cor:BoundedFullDim}, we see that $S^\circ$ is in fact polytopal.
    But then \thref{cor:Hrep-Vrep} implies $V(S^\circ)$ is finite and $S^\circ = \bconv(V(S^\circ))$.
    So, we have
    \eqn{(S^\circ)^\circ = \bigcap_{y\in V(S^\circ)} K(y,1). }
    But since $S$ is a broken line convex set containing the origin, we have $(S^\circ)^\circ=S$ by \thref{prop:DoubleDual}.
\end{proof}










\bibliography{bibliography.bib}
\bibliographystyle{hep.bst}

\end{document}